\documentclass{amsart}

\newtheorem{theorem}{Theorem}[section]
\newtheorem{lemma}[theorem]{Lemma}

\theoremstyle{definition}
\newtheorem{definition}[theorem]{Definition}
\newtheorem{example}[theorem]{Example}

\newtheorem{proposition}[theorem]{Proposition}
\newtheorem{corollary}[theorem]{Corollary}

\theoremstyle{remark}
\newtheorem{remark}[theorem]{Remark}

\numberwithin{equation}{section}

\usepackage{amssymb}
\usepackage{blindtext}
\usepackage{geometry}
 \usepackage{hyperref}
 \hypersetup{hidelinks}

\begin{document}

\title{An Introduction to Complex Game Theory}

\author{Nick Dimou}
\address{Department of Mathematics, National Kapodistrian University of Athens, Greece}
\email{dimou.nikolas@gmail.com}



\begin{abstract}
The known results regarding two-player zero-sum games are naturally generalized in complex space and are presented through a complete compact theory. The payoff function is defined by the real part of the payoff function in the real case, and pure complex strategies are defined by the extreme points of the convex polytope $S_\alpha^m:=\{z\in\mathbb{C}^m:$ $|argz|\leqq\alpha,\;\sum_{i=1}^{m}z_i=1\}$ for ``strategy argument" $\alpha$ in $(0,\frac{\pi}{2})e$. These strategies allow definitions and results regarding Nash equilibria, security levels of players and their relations to be extended in  $\mathbb{C}^{m}$. A new constructive proof of the Minimax Theorem in complex space is given, which indicates a method for precisely calculating the equilibria of two-player zero-sum complex games. A simpler solution method of such games, based on the solutions of complex linear systems of the form $Bz=b$, is also obtained.
\end{abstract}

\maketitle

\section{Introduction}
The concept of extending major definitions and results of Game Theory, specifically of two--player zero--sum games, to the complex space was firstly introduced by G. J. Murray \cite{23}. Soon after, Mond B. and Murray G.J. published a series of papers together (see \cite{19,20}) regarding the generalization of some definitions of Game Theory, the proof of the Minimax Theorem in complex space and the solution of complex matrix games. Their approach was based on the geometric structure and properties of pointed polyhedral closed convex cones in $\mathbb{C}^{m}$. In particular, strategy sets are abstract pointed convex cones that are generated by a finite number of probability vectors. Definitions regarding matrix and symmetric games, as well as some examples of 2x2 complex matrix games and their solutions were also given. Since this first attempt, no other considerable results regarding Game Theory in complex space have followed, with the main and most obvious reasons being: a) this first approach to the natural extension was based on the general form of convex polyhedral cones and that of their corresponding generators--probability vectors; no specific cases of strategy sets were treated, and therefore no other significant new results analogous to those in real space could be found and an exact (algorithmic) solution method for complex matrix games could not be suggested, and b) it was soon realised that an extension of the theory of matrix games in $\mathbb{C}^{m}$ would offer no new results in the context of the existence of game equilibria. Indeed, as it has already been highlighted in older related work (see \cite{10}), if one uses 2n--real space instead of n--complex space then already known results regarding the existance of a Nash equilibrium are easily acquired.\par
In this paper we continue this first work of Mond B. and Murray G.J. and we present a more complete theory of two-player zero-sum games in complex space, which can be characterized as a natural extension of the already known real-space theory.  We achieve this by treating a special case of strategy sets, as they are defined in \cite{19,20}. In particular, strategy sets are defined here by convex polytopes of the form $S_\alpha^m:=\{z\in\mathbb{C}^m:$ $|argz|\leqq\alpha$ and $\sum_{i=1}^{m}z_i=1\}$ with ``strategy argument" $\alpha$ in $(0,\frac{\pi}{2})e$ and with generators--extreme points that represent the generalized form of the extreme points of the standard--probability simplex $\mathbb{P}^m$ of $\mathbb{R}^{m}$. The geometric and optimization properties of these probability vectors allow us to obtain definitions and results, both new and complementary to the ones that appear in the two aforementioned papers. To be more specific, we use such probability vectors to define pure and mixed complex strategies and the mixed extension of a complex matrix game. This separation between the strategies of players, which also appears in the known--real case, is not seen clearly in the previous work and is what appears to be the cornestone of our theory. It is important for the reader to comprehend two direct conclusions: Firstly that the following analysis also holds for other special cases of strategy sets. In fact one can straight away continue the work of Mond B. and Murray G.J. by dealing with the generic form of generated pointed convex cones that appears in their work; it suffices to define pure and complex strategies as done in the following, as well as make accordinate assumptions necessary for the proof of many of the subsequent results (e.g. for the constructive proof of the Minimax Theorem). And secondly that this geometric/game--theoretic approach, that is the use of extreme points of complex convex polytopes as pure strategies of players, can be applied to more generic cases of topological spaces and a complete theory of games can be readily derived.\par 

Most results presented here, are embeded in the already existing real--space theory, considering that the payoff function of such complex games represents the real part of the payoff function of matrix games in the real space, that is $v=Re(x^*Ay)$. However, our analysis also holds for other (real) bilinear payoff functions that correspond to the approach taken, if the appropriate adjustments to the geometric results used are made (e.g. if the payoff function is equal to the imaginary part of the payoff function $u$ of matrix games in $\mathbb{R}^{n}$, then one could use our theory by simply taking the payoff matrix $-iA$ instead of the payoff matrix $A$). Many of the definitions we give derive from the definitions in the real case, but are accordingly formulated and adjusted in order to correspond to the multiplication properties of the new bilinear payoff function and the compact structure of the theory of games in $\mathbb{C}^{m}$ that we wish to present. Same goes with the propositions and main theorems, as we give different proofs of already known results in $\mathbb{R}^{m}$ whose construction goes with the desirable compact theory in complex space. The reader will also notice that all definitions and results are (mathematically) strictly related to a real vector defined as the ``strategy argument". When the strategy argument of a two-player zero-sum complex game is equal to zero (i.e. to the zero vector), then the theory in $\mathbb{C}^{m}$ is identified with the respective theory in $\mathbb{R}^{m}$. Therefore, one could argue that this work represents a natural generalization of the theory of two-player zero-sum games in complex space.\par
The paper is organized as follows: In Section 2 we give some basic geometric notations and known results of linear programming. In Section 3 we prove the two main results of this paper, necessary for the development of the following theory. In particular, through Theorem 3.1 the algebraic and geometric form of all generators of the set $S_\alpha^m:=\{z\in\mathbb{C}^m:$ $|argz|\leqq\alpha$ and $\sum_{i=1}^{m}z_i=1\}$ is obtained, while Theorem 3.2 shows that the strategy set $S_\alpha^m$ is a convex polytope. In Section 4 we restate some definitions given in \cite{19,20}, we give some complementary results of the theory that appears in the latter, we define pure and mixed strategies and we prove that players can improve the value of their security level when they choose to play mixed complex strategies instead of pure ones. In Section 5 we use strong duality for conic linear programming problems in order to give a constructive proof of the Minimax Theorem in complex space. This constructive proof represents a complex-space generalizatized form of the proof given by Dantzig in \cite{6}. Hence, even though the existence of a Nash equilibrium in $\mathbb{C}^{m}$ is already known and the Minimax Theorem has already been proved (see \cite{11,19,20}), a possible solution method is implied for the exact calculation of Nash equilibria of two-player zero-sum complex games. In Section 6 we re-state the definition of strategy domination as it is already known and we generalize the process of iterated elimination via the use of the generators--pure strategies of the strategy set $S^{m}_{\alpha}$. In Section 7 we show that the solution of two-player zero-sum games is also based on the solution of complex linear systems of the form $Bz=b$ with the use of complex equalizing strategies, similar to the real case. These strategies satisfy the complementarity conditions $Re(y,b)-Re(y,Ax)=0$, $Re(x,c)-Re(x,A^*y)=0$ and therefore agree with the constructive proof of the Minimax Theorem in Section 5. In Section 8 the elimination process, equalizing strategies and algorithmic calculation methods derived from the proof of the Minimax Theorem are used in order to describe a step-by-step solution method for complex matrix games. Lastly, in Section 9 we define ``games of common argument" and we re-define symmetric games in the complex case.

\[\]\[\]
\section{Geometric Preliminaries}
In this section some basic notations and some important geometric and linear programming results are given.\\

Let $S\subseteq\mathbb{C}^m$. The dual cone of $S$, denoted $S^*$, is the set: $S^*=\{w\in\mathbb{C}^m: Re(z,w)\ge0 \;\; \forall z\in S\}$. Note that $S$ is a closed convex cone\footnote[1]{A set $S\subseteq\mathbb{C}^m$ is a convex cone if and only if $\lambda S\subseteq S$ and $S+S\subseteq S$.} if and only if $S=(S^*)^*$. We define the dual cone by the function $Re(\cdot,\cdot)$ since we only deal with this payoff function in our analysis.\par
A point $a\in S$ is called an extreme point if $\nexists b,c\in S$ and $\nexists\lambda\in(0,1)$ such that $b\not=c$ and $\lambda b+(1-\lambda)c=a$. That is, a point $a\in S$ is an extreme point if it cannot be written as a convex combination of points not equal to $a$. We denote $\mathcal{A}(S)$ for the set of extreme points of $S$.\par
Let $A\subseteq\mathbb{C}^m$. The smallest convex set in $\mathbb{C}^m$ that contains $A$ is called the convex hull of $A$. We denote $co(A)$ or $conv(A)$ for the convex hull of $A$. The following theorem holds:
\[\]
\begin{theorem}\label{th:2.6}
Let $S,\;A\subseteq\mathbb{C}^{m}$. Then $S=co(A)$ if and only if S contains all convex combinations of the points of $A$ and no other points.
\end{theorem}

A set $S\subseteq\mathbb{C}^m$ is called a convex polytope if (i) it is convex, (ii) it contains a finite number $n$ of (extreme) points $z_1,z_2,...,z_n$ such that $S=co(\{z_1,z_2,...,z_n\})$. We say that $S$ is generated by the points $z_1,z_2,...,z_n$.\par

There is often a big confusion between the definitions of convex polytopes and convex polyhedra in convex geometry and linear programming. In the more recent bibliography, and as used here, convex polyhedra represent the general case of unbounded polytopes and convex polytopes represent the bounded form of convex polyhedra. In the old bibliography this differentiation is usually reversed. We also write ``(extreme)" in the above definition, simply because in this paper we only deal with convex polytopes that are generated by extreme points. Therefore from now on, without loss of generality, when referring to a convex polytope $S\subseteq\mathbb{C}^{m}$ we assume that it is generated by the extreme points that belong to it (note that in this case, by Theorem \ref{th:2.6}, there cannot be an extreme point $z_0$ that belongs to $S$ but doesn't belong to $\{z_1,z_2,...,z_n\}$).\par
We have the subsequent two results regarding convex polytopes:

\begin{proposition}\label{pr:2.11}
Every convex polytope $S$ in $\mathbb{C}^m$ is compact.
\end{proposition}
\begin{proof}
Assume $S=co(\{z_1,z_2,...,z_n\})$ where $n\geq2$ (for $n=0$ or $1$ we have nothing to show) and let $f:\mathbb{P}^n\rightarrow\mathbb{C}^m$ be a function such that $f(\lambda)=\lambda_1z_1+\lambda_2z_2+...+\lambda_nz_n$, where $\mathbb{P}^n=\{\lambda\in\mathbb{R}^n:\lambda_i\ge0$ $\forall i$ and $\displaystyle\sum_{i=1}^{n}\lambda_i=1\}$ is the standard simplex of $\mathbb{R}^n$. Then $f$ is obviously continuous and $f(\mathbb{P}^n)=S$ by Theorem \ref{th:2.6}. Hence $S$ is compact, since $\mathbb{P}^n$ is compact.
\end{proof}

\begin{theorem}\label{th:2.16}
Let $S\subseteq\mathbb{C}^m$ be a convex polytope and $H(c,\alpha):=\{z\in\mathbb{C}^{m}:Re(c,z)=\alpha\}$ be a supporting hyperplane of $S$. Then there exists an extreme point $z_0$ of $S$ that belongs to $H(c,\alpha)$.
\begin{proof}
Let us assume that $S=co(\{z_1,...,z_n\})$ $(n\geq2)$ and suppose that there are no extreme points of $S$ that belong to $H(c,\alpha)$. Let $w\in S\cap H(c,\alpha)$ (such point exists since $S$ is closed). Since $S$ is a convex polytope by Theorem \ref{th:2.6} $\exists \lambda_1,\lambda_2,...,\lambda_n\geq 0$ such that $\displaystyle\sum_{i=1}^{n}\lambda_i=1$ and $w=\lambda_1z_1+\lambda_2z_2+...+\lambda_nz_n$ $\Rightarrow Re(c^*w)=\lambda_1Re(c^*z_1)+\lambda_2Re(c^*z_2)+...+\lambda_nRe(c^*z_n)$. Since $Re(c^*w)=\alpha$ and $Re(c^*z_i)>\alpha$ $\forall i\in\{1,2,...,n\}$, we obtain $\alpha>\alpha$ which is a contradiction.\\
\end{proof}
\end{theorem}

\begin{theorem}\label{th:3.12}
Let $c\in\mathbb{C}^{m}$ and $S\subseteq\mathbb{C}^{m}$ be a convex polytope. Now consider the following problem:
\begin{center}
\begin{equation*}{(I)\;\;\;\;\;\;}
\begin{array}{ll}
    min\;\; Re(c^*z)\\
    s.t.\;\; z\in S
    \end{array}
    \end{equation*}
    \end{center}
Then:
\begin{enumerate}
    \item $(I)$ has an optimal feasible solution.
    \item if $z^0$ is an optimal feasible solution of $(I)$, then $z^0\in\partial S$.
    \item if $H(c,a)$ is the hyperplane that corresponds to an optimal feasible solution $z^0$, that is if $H(c,a)=H(c,Re(c^*z^0))$, then $H(c,a)$ supports $S$. 
\end{enumerate}
\end{theorem}
\begin{proof}
\begin{enumerate}
    \item By Proposition \ref{pr:2.11} $S$ is compact. Hence $(I)$ has an optimal feasible solution.
    \item Since $S$ is closed, it suffices to show that $z^0\not\in S^{\mathrm{o}}$. Suppose that $z^0\in S^{\mathrm{o}}$. Then $\exists\varepsilon>0:$ $B(z^0,\varepsilon)\subset S$. In fact, one can take $\varepsilon>0$ such that $\hat{z}:=z^0-\displaystyle\frac{\varepsilon}{2}\displaystyle\frac{c}{(\|c\|)^2}\in B(z^0,\varepsilon)\Rightarrow \hat{z}\in S$. Then
    \begin{equation*}
        \hat{z}=z^0-\frac{\varepsilon}{2}\frac{c}{(\|c\|)^2}\Rightarrow c^*\hat{z}=c^*z^0-\frac{\varepsilon}{2}\cdot1\Rightarrow Re(c^*\hat{z})=Re(c^*z^0)-\frac{\varepsilon}{2}<Re(c^*z^0)
    \end{equation*}
    which is a contradiction since $z^0$ is optimal feasible.
    \item Let $a:=Re(c^*z^0)$ and $H(c,a)$ the hyperplane that corresponds to $z^0$. Then it is obvious that $Re(c^*z)\geq a$ $\forall z\in S$ since $z^0$ is optimal feasible, hence by (2) $H(c,a)$ supports $S$.
\end{enumerate}
\end{proof}
We now obtain the following important result:
\begin{corollary}\label{cor:pi}
For a given convex polytope $S$ consider the problem $(I)$ as defined in the previous theorem. Then there exists at least one $z^0\in\mathcal{A(S)}:$ $z^0$ is an optimal feasible solution of $(I)$. 
\end{corollary}
\begin{proof}
Let $z^0$ be an optimal solution of $(I)$ and $H(c,a):=H(c,Re(c,z^0))$ be the corresponding supporting hyperplane of $S$. Then by Theorems \ref{th:2.16} and \ref{th:3.12} \;$\exists z^0\in H(c,a)\cap\mathcal{A(S)}$.
\end{proof}
\begin{remark}
Theorem \ref{th:3.12} and Corollary \ref{cor:pi} also hold for the problem:
\begin{center}
\begin{equation*}{(II)\;\;\;\;\;\;}
\begin{array}{ll}
    max\;\; Re(c^*z)\\
    s.t.\;\; z\in S
    \end{array}
    \end{equation*}
    \end{center}
\end{remark}
\begin{definition}(Complex Linear Complementarity Problem)\\
Let $M\in\mathbb{C}^{p\times p}$, $q\in\mathbb{C}^{p}$, $\gamma\in\mathbb{R}^{p}$ such that $0\leqq\gamma\leqq\frac{\pi}{2}e$. The complex linear complementarity problem (complex LCP) is stated as:
Find $x\in\mathbb{C}^{p}$ and $y=q+Mx\in\mathbb{C}^{p}$ such that
\begin{equation}
    \lvert arg (x) \rvert \leqq\gamma,\;\;\lvert arg(y)\rvert \leqq\frac{\pi}{2}e-\gamma 
\end{equation}
and 
\begin{equation}
    Re(x^*y)=0\\
\end{equation}

\end{definition}
An existance theory and a solution method for the complex LCP were first studied by C.J. McCallum Jr. \cite{13,14}.

\[\]
\section{Strategy Sets}
We deal with complex matrix games where players have strategy sets of the form $S_\alpha^m:=\{z\in\mathbb{C}^m:$ $|argz|\leqq\alpha$ and $\sum_{i=1}^{m}z_i=1\}$ for some $m\in\mathbb{N}$ and $\alpha\in(0,\frac{\pi}{2})e$. A theorem regarding the extreme points of the set $S^{m}_\alpha\subseteq\mathbb{C}^{m}$ is given. These points represent the generalized form of the extreme points of the standard simplex $\mathbb{P}^{m}$ of $\mathbb{R}^{m}$, and therefore that of probability vectors, in complex space. It is the cornestone of our geometric approach towards the extension of Game Theory to the complex space, alongside the next theorem which states that the set $S^{m}_\alpha$ is a convex polytope.
\begin{theorem}\label{th:2.28}
Let $S_\alpha^m:=\{z\in\mathbb{C}^m:$ $|argz|\leqq\alpha$ and $\displaystyle\sum_{i=1}^{m}z_i=1\}$, where $\alpha\in\mathbb{R}^m$ is of the form $\alpha=\alpha_0e$ for some $\alpha_0\in(0,\frac{\pi}{2})$ and $e=(1,1,...,1)^T$. Then $S_\alpha^m$ is convex and $\mathcal{A}(S^m_\alpha)=\{d^i:i=1,2,...,m^2\}$ where:
\begin{equation*}
    d^i=\left\{
\begin{array}{ll}
      e^i & i=1,2,...,m \\
      \eta^i & i=m+1,m+2,...,m^2\\
\end{array} 
\right. 
\end{equation*}
\[\]
where $e^i=(0,...,0,1,0,...,0)^T$ $(e^i_i=1, e^i_j=0$ $\forall j\neq i)$ and
\begin{equation*}
    \eta^{m+1}=(\frac{1}{2}+bi,\frac{1}{2}-bi,0,...,0)^T,\eta^{m+2}=(\frac{1}{2}+bi,0,\frac{1}{2}-bi,0,...,0)^T_{,}...,  \eta^{2m-1}=(\frac{1}{2}+bi,0,..,0,\frac{1}{2}-bi)^T
\end{equation*}
\begin{equation*}
    \eta^{2m}=(\frac{1}{2}-bi,\frac{1}{2}+bi,0,...,0)^T_{,}\eta^{2m+1}=(0,\frac{1}{2}+bi,\frac{1}{2}-bi,0,...,0)^T_{,}...,\eta^{3m-2}=(0,\frac{1}{2}+bi,0,...,0,\frac{1}{2}-bi)^T
\end{equation*}
\begin{equation*}
    \vdots
\end{equation*}
\begin{equation*}
    \eta^{(m-1)m+2}=(\frac{1}{2}-bi,0,...,0,\frac{1}{2}+bi)^T,\;\eta^{(m-1)m+3}=(0,\frac{1}{2}-bi,0,...,0,\frac{1}{2}+bi)^T,\;...\;,\;\eta^{m^2}=
\end{equation*}
$=(0,...,0,\displaystyle\frac{1}{2}-bi,\displaystyle\frac{1}{2}+bi)^T$, for $\frac{1}{2}\pm{bi}\in\mathbb{C}$ such that $arg(\frac{1}{2}+bi)=\alpha_0$ and $arg(\frac{1}{2}-bi)=-\alpha_0$\\
\begin{proof}
Let $w,h\in S^m_\alpha$ and $\lambda\in(0,1)$. Then $|arg\lambda w|\leqq\alpha$ and $|arg(1-\lambda)h|\leqq\alpha$ $\Rightarrow$ $|arg[\lambda w+(1-\lambda)h]|\leqq\alpha$ (since $a\in(0,\frac{\pi}{2})$). We also have
\begin{equation*}
    \displaystyle\sum_{i=1}^{m}[\lambda w_i+(1-\lambda)h_i]=\lambda\displaystyle\sum_{i=1}^{m}w_i+(1-\lambda)\displaystyle\sum_{i=1}^{m}h_i=\lambda+(1-\lambda)=1
\end{equation*}
and therefore $\lambda w+(1-\lambda)h\in S^m_\alpha$
\[\]
Regarding the extreme points of $S^m_\alpha$ one can easily see that all $e^i=(0,...,0,1,0,...,0)^T$ $(e^i_i=1, e^i_j=0$ $\forall j\neq i)$ are extreme points: similarly to the real case, $\nexists z,w\in S^{m}_\alpha$ and $\nexists\lambda\in(0,1)$ such that $(0,...,0,1,0,...,0)=\lambda z+(1-\lambda)w$.\\
Note that every vector of the form $(x_1,x_2,...,x_m)\not=e^k$ in $S^{m}_\alpha$ where $x_1,x_2,...,x_m\in[0,1]$ is not an extreme point, since it can be written as a convex combination of the extreme points $e^1,e^2,...,e^m$.\\ We now prove the following necessary condition:\\ \\
Proposition 1: Let $z\neq e^k$ $\forall k$ be an extreme point of $S^m_\alpha$. Then for every $z_i$ such that $Re(z_i)\in(0,1)$ the following necessarily holds: $|arg(z_i)|=\alpha_0$\\ \\
By the above remark regarding points with coordinates that have an imaginary part equal to zero, we prove the proposition for vectors that have at least two (since $\displaystyle\sum_{i=1}^{m}z_i=1)$ coordinates with a non-zero imaginary part. We show that if $\exists z_i$ with $Re(z_i)\in(0,1)$ such that $|arg(z_i)|<\alpha_0$, then $z$ is not an extreme point. We have the following cases:
\begin{enumerate}
    \item Assume that there is exactly one coordinate $z_i$ of $z$ such that $Re(z_i)\in(0,1)$ and $|arg(z_i)|<\alpha_0$. Let $z_k=a+bi$ be this coordinate for some $k$. We may assume (without loss of generality) that $b\geq0$ (the case $b\leq0$ is similarly proved). We take the following subcases:
    \begin{enumerate}
        \item Assume that there exists $z_j=c+di$ with $c\in(0,1)$, $d>0$ for some $j\neq k$. Then $arg(z_j)=\alpha_0$. One can find $x_0,y_0>0$ ($x_0,y_0<<a,b,c,d$, assuming $b>0$) such that $\displaystyle\frac{d}{c}=\displaystyle\frac{y_0}{x_0}$ and $|arg[(a\pm x_0)+(b\pm y_0)i]|\leq\alpha_0$: Since $z_k\not\in\partial S^m_\alpha$ one may simply take the open ball $B(z_k,\varepsilon)$ centered at $z_k$ with (very small) radius $\varepsilon>0$ ($\varepsilon<<a,b,c,d$, if $b>0$) such that $\forall w\in B(z_k,\varepsilon):$ $0\leq |argw|\leq\alpha_0$. Then, if $c\leq d$ let $y_0:=\varepsilon/2^{k+2}$ and $x_0:=\displaystyle\frac{y_0c}{d}$, while if $c\geq d$ let $x_0:=\varepsilon/2^{k+2}$ and $y_0:=\displaystyle\frac{dx_0}{c}$. Now notice that $\displaystyle\frac{d-y_0}{c-x_0}=\displaystyle\frac{d}{c}=\displaystyle\frac{d+y_0}{c+x_0}$, hence $arg[(c+x_0)+(d+y_0)i]=arg[(c-x_0)+(d-y_0)i]=\alpha_0$. Let $w,h\in\mathbb{C}^m$ such that $w_i=h_i=z_i$ $\forall i\neq k,j$, $w_k=(a-x_0)+(b-y_0)i$, $w_j=(c+x_0)+(d+y_0)i$ and $h_k=(a+x_0)+(b+y_0)i$, $h_j=(c-x_0)+(d-y_0)i$. From the above one can see that $w,h\in S^m_\alpha$ and $z=\frac{1}{2}w+\frac{1}{2}h$. Therefore $z$ is not an extreme point of $S^m_\alpha$.
        \item Assume that $\forall i\neq k:$ $Re(z_i)>0\Rightarrow Im(z_i)\leq0$. Let $z_j=c+di$ with $c\in(0,1)$, $d<0$ and $arg(z_j)=-\alpha_0$. Then, exactly as shown above, we can find $x_0,y_0>0$ such that $\displaystyle\frac{d}{c}=-\frac{y_0}{x_0}$ and $|arg[(a\pm x_0)+(b\mp y_0)i]|\leq\alpha_0$. Let $w,h\in\mathbb{C}^m$ such that $w_i=h_i=z_i$ $\forall i\neq k,j$, $w_k=(a+x_0)+(b-y_0)i$, $w_j=(c-x_0)+(d+y_0)i$ and $h_k=(a-x_0)+(b+y_0)i$, $h_j=(c+x_0)+(d-y_0)i$. Then $w,h\in S^m_\alpha$ and $z=\frac{1}{2}w+\frac{1}{2}h$. Therefore $z$ is not an extreme point of $S^m_\alpha$.
    \end{enumerate}
    \item Assume that the number of coordinates  $z_i$ of $z$ with $Re(z_i)\in(0,1)$, $|arg(z_i)|<\alpha_0$ is at least $2$. Let $z_i,z_j$ be two of those coordinates of $z$, $z_i:=a+bi$, $z_j:=c+di$
    \begin{enumerate}
        \item Assume that both $z_i,z_j$ have an imaginary part of the same sign or zero. We can assume that $Im(z_i),Im(z_j)\geq0$ (the other case is similarly proved). Then since $arg(z_i),arg(z_j)<\alpha_0$ we can find $x_0,y_0>0$ ($x_0,y_0<<a,b,c,d$, assuming $b,d>0$) similarly as above such that $|arg[(a\pm x_0)+(b\pm y_0)i]|,|arg[(c\pm x_0)+(d\pm y_0)i]|\leq\alpha_0$ (let $y_0$, $x_0$ be equal to $1/2^{i+j+2}$ times the minimum of the (very small) radiuses of the open balls centered at $z_i$ and $z_j$ respectively). Let $w,h\in\mathbb{C}^m$ such that $w_k=h_k=z_k$ $\forall k\neq i,j$, $w_i=(a-x_0)+(b-y_0)i$, $w_j=(c+x_0)+(d+y_0)i$ and $h_i=(a+x_0)+(b+y_0)i$, $h_j=(c-x_0)+(d-y_0)i$. Then $w,h\in S^m_\alpha$ and $z=\frac{1}{2}w+\frac{1}{2}h$. Therefore $z$ is not an extreme point of $S^m_\alpha$.
        \item Assume that $z_i,z_j$ have an imaginary part of different sign or zero. We can assume that $Im(z_i)\geq0$ and $Im(z_j)\leq0$ (the other case is similarly proved). Then since $0\leq arg(z_i)<\alpha_0$ and $0\geq arg(z_j)>-\alpha_0$ one can find $x_0,y_0>0$ ($x_0,y_0<<a,b,c,-d$, for $b,-d>0$) as above such that $|arg[(a\pm x_0)+(b\mp y_0)i]|,|arg[(c\pm x_0)+(d\mp y_0)i]|\leq\alpha_0$ Now let $w,h\in\mathbb{C}^m$ such that $w_k=h_k=z_k$ $\forall k\neq i,j$, $w_i=(a+x_0)+(b-y_0)i$, $w_j=(c-x_0)+(d+y_0)i$ and $h_i=(a-x_0)+(b+y_0)i$, $h_j=(c+x_0)+(d-y_0)i$. Then $w,h\in S^m_\alpha$ and $z=\frac{1}{2}w+\frac{1}{2}h$. Therefore $z$ is not an extreme point of $S^m_\alpha$.
    \end{enumerate}
\end{enumerate}
We now prove a second proposition:\\ \\
Proposition 2: Let $z\neq e^k$ $\forall k$ be an extreme point of $S^m_\alpha$. Then there is no pair of coordinates of $z$ with a nonzero imaginary part of the same sign.\\ \\
Suppose that there exist $z_i:=a+bi$ and $z_j:=c+di$ with $b,d>0$ (the case $b,d<0$ is similarly proved). By Proposition 1 we obtain $arg(z_i)=arg(z_j)=\alpha_0$ and therefore $\displaystyle\frac{b}{a}=\displaystyle\frac{d}{c}$. One can now easily find $x_0,y_0>0$ ($x_0,y_0<<a,b,c,d)$ such that $\displaystyle\frac{y_0}{x_0}=\displaystyle\frac{b}{a}=\displaystyle\frac{d}{c}$: let $y_0:=b\varepsilon$, $x_0:=a\varepsilon$ for some $\varepsilon>0$. Let $w,h\in\mathbb{C}^m$ such that $w_k=h_k=z_k$ $\forall k\neq i,j$, $w_i=(a-x_0)+(b-y_0)i$, $w_j=(c+x_0)+(d+y_0)i$ and $h_i=(a+x_0)+(b+y_0)i$, $h_j=(c-x_0)+(d-y_0)i$. Then $w,h\in S^m_\alpha$ and $z=\frac{1}{2}w+\frac{1}{2}h$. Therefore $z$ is not an extreme point of $S^m_\alpha$, which is a contradiction.\\ \\
Note that in all the above cases where we take two complex numbers with an imaginary part of the same sign, their real part is $<\frac{1}{2}$ since $z\in S^{m}_\alpha$ ($\leq\frac{1}{2}$ in the case where the two complex numbers have an imaginary part of a different sign), and therefore the new vectors $w,h$ are well defined with the corresponding real parts being $\leq\frac{1}{2}$, and their arguments being suitable for the respective result in each case to hold.\\ \\
From the above propositions and the remark made before them we obtain the following: If $z\neq e^k$ $\forall k$ is an extreme point of $S^m_\alpha$ then it only has two nonzero coordinates that necessarily are of the form $\frac{1}{2}+bi$, $\frac{1}{2}-bi$ for $b>0:$ $arg(\frac{1}{2}+bi)=\alpha_0$ and $arg(\frac{1}{2}-bi)=-\alpha_0$.\\ \\
One can now easily show that all points $\eta^k$ for $k\in\{{m+1,...,m^2}\}$, as defined in the formulation of the theorem, are indeed extreme points of $S^m_\alpha$. For the completeness of the proof:\\ \\
Suppose $z=(0,...,0,\frac{1}{2}\pm bi,0,...,0,\frac{1}{2}\mp bi,0,...,0)^T\in S^m_\alpha$ is not an extreme point. Then there exist $z\neq w,h\in S^m_\alpha$ and $\lambda\in(0,1)$ such that $z=\lambda w+(1-\lambda)h$. Let $z_i,z_j$ be the nonzero coordinates of $z$. Since $\lambda$, $(1-\lambda)>0$ we obtain that $w_k=h_k=0$ $\forall k\neq i,j$ and 
\begin{equation}
    \lambda w_i+(1-\lambda)h_i=\frac{1}{2}\pm bi
\end{equation}  
\begin{equation}
    \lambda w_j+(1-\lambda) h_j=\frac{1}{2}\mp bi
\end{equation}
Suppose $Re(w_i)<\frac{1}{2}\Rightarrow Re(w_j)>\frac{1}{2}$, $Re(h_i)>\frac{1}{2}$, $Re(h_j)<\frac{1}{2}$ and perforce, since $w_i+w_j=1$ and $|arg(w_i)|\leq\alpha_0$, we have $|Im(w_i)|,|Im(w_j)|,|Im(h_i)|,|Im(h_j)|<b$ $\Rightarrow |\lambda Im(w_i)+(1-\lambda)Im(h_i)|<b$, which is a contradiction.\\ \\
Similarly, $Re(w_i)\not>\frac{1}{2}$, hence $Re(w_i)=\frac{1}{2}$, and similarly $Re(w_j)=Re(h_i)=Re(h_j)=\frac{1}{2}$ and $Im(w_i)=Im(h_i)=\pm b$, $Im(w_j)=Im(h_j)=\mp b$ $\Rightarrow w=h=z$.
\\ \\
Therefore $\mathcal{A}(S^m_\alpha)=\{d^i:i={1,2,...,m^2}\}$ where $d^i$ as defined in the formulation of the theorem above.
\end{proof}
\end{theorem}
\[\]
\begin{theorem}
The set $S^m_\alpha$ for some $\alpha=\alpha_0e$, $\alpha_0\in(0,\frac{\pi}{2})$ is a convex polytope.
\begin{proof}
It is already known by Theorem \ref{th:2.28} that $S^m_\alpha$ is convex and $|\mathcal{A}(S^m_\alpha)|=m^2<\infty$. By Theorem \ref{th:2.6} it suffices to show that every $z\in S^m_\alpha$ can be written as a convex combination of the extreme points of $S^m_\alpha$.\\ \\Let $z\not=e^j\in S^m_\alpha$. Note that $\forall i\in\{1,...,m\}$ $z_i$ can be written as: $z_i=\lambda_i(\frac{1}{2}+bi)+\mu_i(\frac{1}{2}-bi)+\kappa_i\cdot1$ for some $\lambda_i,\mu_i,\kappa_i\in[0,1]$, $|arg(\frac{1}{2}\pm bi)|=\alpha_0$ (note that $\lambda_i,\kappa_i,\mu_i\not\in(1,\infty)$ since $|arg(z_i)|\leq\alpha_0$ and $\displaystyle\sum_{i=1}^{m}z_i=1)$ such that
\begin{equation*}
    \displaystyle\sum_{i=1}^{m}\lambda_i=\displaystyle\sum_{i=1}^{m}\mu_i,\;\; \displaystyle\sum_{i=1}^{m}(\frac{\lambda_i}{2}+\frac{\mu_i}{2}+\kappa_i)=\displaystyle\sum_{i=1}^{m}(\lambda_i+\kappa_i)=1
\end{equation*} 
Indeed:
\begin{enumerate}
    \item If $Im(z_i)>0$ then $\mu_i:=0$. Moreover, If $argz_i=\alpha_0$ then $\kappa_i:=0$ and $z_i=\lambda_i(\frac{1}{2}+bi)$ for some $\lambda_i\in(0,1]$, while if $argz_i<\alpha_0$ then $z_i$ can be written as $\lambda_i(\frac{1}{2}+bi)+\kappa_i$ for some $\lambda_i,\kappa_i>0$ such that $\lambda_i/2+\kappa_i<1$ (we subtract a real number $\kappa_i$ such that the complex number $z_i-\kappa_i$ has an argument of $\alpha_0$).
    \item If $Im(z_i)<0$ then $\lambda_i:=0$. Moreover, If $argz_i=-\alpha_0$ then $\kappa_i:=0$ and $z_i=\mu_i(\frac{1}{2}-bi)$ for some $\mu_i\in(0,1]$, while if $argz_i>-\alpha_0$ then $z_i$ can be written as $\mu_i(\frac{1}{2}-bi)+\kappa_i$ for some $\mu_i,\kappa_i>0$ such that $\mu_i/2+\kappa_i<1$.
    \item If $Im(z_i)=0$ then $z_i=\kappa_i\in[0,1)$ and $\lambda_i,\mu_i:=0$.
\end{enumerate}
Following the above, $z$ can be written as:
\begin{equation*}
    (\lambda_1(\frac{1}{2}+bi)+\mu_1(\frac{1}{2}-bi),\lambda_2(\frac{1}{2}+bi)+\mu_2(\frac{1}{2}-bi),...\;,\lambda_m(\frac{1}{2}+bi)+\mu_m(\frac{1}{2}-bi))^T+\kappa_1e^1+\kappa_2e^2+...+\kappa_me^m
\end{equation*}
It suffices to show that the vector $w:=(\lambda_1(\frac{1}{2}+bi)+\mu_1(\frac{1}{2}-bi),\lambda_2(\frac{1}{2}+bi)+\mu_2(\frac{1}{2}-bi),...\;,\lambda_m(\frac{1}{2}+bi)+\mu_m(\frac{1}{2}-bi))^T$ can be written as $\displaystyle\sum_{i=1}^{\nu}c_i\eta^{j_i}$ for some $j$ with index $1\leq\nu\leq m^2$ and $c_i\in(0,1)$ $\forall i$ such that $\displaystyle\sum_{i=1}^{\nu}c_i=\displaystyle\sum_{i=1}^{m}\lambda_i$. Note that in $w$ each coordinate has $\lambda_i=0$ or $\mu_i=0$ or both.\\ \\
Now assume that the number of coordinates of $w$ with positive imaginary part is $k$, the number of those with negative imaginary part is $l$ and the number of those with imaginary part equal to zero is $n$. Obviously $k+l+n=m$. Let us assume that $k,l>1$ (the trivial case represents a trivial extreme point $e^j$). Consider the sets
\begin{equation*}
    A_1:=\{p_1,p_2,...\;,p_k\},\;\;B_1:=\{q_1,q_2,...\;,q_l\}
\end{equation*}
where $p_1,p_2,...,p_k$ represent the coefficients $\lambda_i$ of the coordinates of $w$ with positive imaginary parts, so that (without loss of generality) $p_1$ is the coefficient $\lambda_{i_1}$ for some $i_1$ of the first coordinate of $w$ (counting from the left) that appears with positive imaginary part, $p_2$ is the coefficient of the second coordinate that appears with positive imaginary part,..., $p_k$ is the coefficient of the $k-$coordinate of $w$ that appears with positive imaginary part. Similarly, $q_1,q_2,...,q_l$ represent the counting of coefficients $\mu_i$ of the coordinates of $w$ (without loss of generality) that appear with negative imaginary part. The index of each set is equal to the index of the first point that belongs to it. Also note that
\begin{equation*}
\displaystyle\sum_{i=1}^{k}p_i=\displaystyle\sum_{i=1}^{m}\lambda_i=\displaystyle\sum_{i=1}^{m}\mu_i=\displaystyle\sum_{i=1}^{l}q_i.
\end{equation*}\\
We now take the difference between the first points of each set, and we deal with three different cases:
\begin{enumerate}
    \item If $p_1-q_1>0$ then let $r_1:=p_1-q_1$. By these coefficients $p_1$ and $q_1$ we can ``extract" from $w$ the vector $(0,...,0,q_1(\frac{1}{2}\pm bi),0,...,0,q_1(\frac{1}{2}\mp bi),0,...,0)^T$, where the nonzero coordinates correspond to the coordinates of $w$ that have $p_1$ and $q_1$ as coefficients respectively. Note that
    \begin{equation*}
        (0,...,0,q_1(\frac{1}{2}\pm bi),0,...,0,q_1(\frac{1}{2}\mp bi),0,...,0)^T=c_1\eta^{j_1}
    \end{equation*}
    for some $j_1$ where $c_1=q_1$. That is, the vector $w$ can be written as a sum of two vectors $h_1$ and $c_1\eta^{j_1}$, where $h_1:=w-c_1\eta^{j_1}$. We then consider the sets 
    \begin{equation*}
        A_1':=\{r_1,p_2,...,p_k\},\;\;B_2:=\{q_2,q_3,...,q_l\}
    \end{equation*}
    \item If $p_1-q_1<0$ then let $r_1:=q_1-p_1$. By these coefficients $p_1$ and $q_1$ we can ``extract" from $w$ the vector $(0,...,0,p_1(\frac{1}{2}\mp bi),0,...,0,p_1(\frac{1}{2}\pm bi),0,...,0)^T$, where the nonzero coordinates correspond to the coordinates of $w$ that have $q_1$ and $p_1$ as coefficients respectively. Note that
    \begin{equation*}
        (0,...,0,p_1(\frac{1}{2}\mp bi),0,...,0,p_1(\frac{1}{2}\pm bi),0,...,0)^T=c_1\eta^{j_1}
    \end{equation*}
    for some $j_1$ where $c_1=p_1$. That is, the vector $w$ can be written as a sum of two vectors $h_1$ and $c_1\eta^{j_1}$, where $h_1:=w-c_1\eta^{j_1}$. We then consider the sets 
    \begin{equation*}
        A_2:=\{p_2,p_3,...,p_k\},\;\;B_1':=\{r_1,q_2,...,q_l\}
    \end{equation*}
    \item If $p_1-q_1=0$, then by these coefficients $p_1$ and $q_1$ we can ``extract" from $w$ the vector $(0,...,0,p_1(\frac{1}{2}\mp bi),0,...,0,p_1(\frac{1}{2}\pm bi),0,...,0)^T$, where the nonzero coordinates correspond to the coordinates of $w$ that have $q_1$ and $p_1$ as coefficients respectively. Note that
    \begin{equation*}
        (0,...,0,p_1(\frac{1}{2}\mp bi),0,...,0,p_1(\frac{1}{2}\pm bi),0,...,0)^T=c_1\eta^{j_1}
    \end{equation*}
    for some $j_1$ where $c_1=p_1$. That is, the vector $w$ can be written as a sum of two vectors $h_1$ and $c_1\eta^{j_1}$, where $h_1:=w-c_1\eta^{j_1}$. We then consider the sets 
    \begin{equation*}
        A_2:=\{p_2,p_3,...,p_k\},\;\;B_2:=\{q_2,q_3,...,q_l\}
    \end{equation*}
\end{enumerate}
Note that in each case the sum of the points of the new set $A$ is equal to the sum of the points of the new set $B$. Regardless of which of the three cases holds, we repeat the same process: we consider the difference of the first elements of the two new sets ($A_2$, $B_1'$, or $A_1'$, $B_2$, or $A_2$, $B_2$) and we ``extract" from the vector $h_1$ an extreme point $\eta^{j_2}$ with coefficient $c_2$ the smaller of the two elements of the sets. Since the initial sets $A_1$ and $B_1$ are finite and all coefficients belong to $(0,1)$, this process will also be finite (one can easily prove this via proof by contradiction) and will end when at least one of the new defined sets is the empty set. Observe that during this process, the two sums of the elements of the two new sets defined each time are equal, and that the two sums of the elements that have been ``eliminated" from the initial sets $A_1$, $B_1$ are also equal, due to the relation $\displaystyle\sum_{i=1}^{k}p_i=\displaystyle\sum_{j=1}^{l}q_j$. To be more specific, after the final step of the process, we will end up with exactly one of the following pairs:
\begin{enumerate}
    \item $A_\xi'=\{r_\xi,p_{\xi+1},...,p_k\}$ and $B_{l+1}=$\O,\;\;for some $\xi\leq k$, or
    \item $A_{\xi+1}=\{p_{\xi+1},p_{\xi+2},...,p_k\}$ and $B_{l+1}=$\O,\;\;for some $\xi\leq k-1$, or
    \item $A_{k+1}=$\O\;\;and $B_\xi'=\{r_\xi,q_{\xi+1},...,q_l\}$, for some $\xi\leq l$, or
    \item $A_{k+1}=$\O\;\;and $B_{\xi+1}=\{q_{\xi+1},q_{\xi+2},...,q_l\}$, for some $\xi\leq l-1$, or
    \item $A_{k+1}=$\O\;\;and $B_{l+1}=$\O\\
\end{enumerate}
We can now see by the above observation that only the last case holds, and therefore the vector $w$ is written as a sum $\displaystyle\sum_{i=1}^{\nu}c_i\eta^{j_i}$ for some $j$ with index $1\leq\nu\leq m^2$ and $c_i\in(0,1)$ $\forall i$ such that $\displaystyle\sum_{i=1}^{\nu}c_i=\displaystyle\sum_{i=1}^{k}p_i=\displaystyle\sum_{i=1}^{m}\lambda_i$ (the final ``non-extracted" vector $h_\nu=0e$):
Suppose that e.g. the second case holds (the other three cases are similarly proved). Then by the above observation $\displaystyle\sum_{j=1}^{l}q_j=\displaystyle\sum_{i=1}^{\xi}p_i\Rightarrow\displaystyle\sum_{i=1}^{m}\mu_i=\displaystyle\sum_{j=1}^{l}q_j<\displaystyle\sum_{i=1}^{k}p_i=\displaystyle\sum_{i=1}^{m}\lambda_i$, which is a contradiction.
\end{proof}
\end{theorem}
\[\]

\section{Complex strategies, Security levels and Nash equilibria}
In this section primary definitions of pure and mixed complex strategies and security levels are given. Theorems regarding the relation between the security level of different players in the case of two-player zero-sum complex games follow.\\\par
In this particular analysis of our theory we assume that the players have the same ``game-role" and are of the same ``type" as the ones in the known (simple) real case of decision making. That is, we assume that the condition of common knowledge is met, the soon to be defined complex games are games with complete information, and that each player is an intelligent, rational, strategic decision--maker within the context of the game that acts accordingly and independently in order to maximize their payoff function. We also assume that every player's preference relation over the set of all different forms of outcomes satisfies the von Neumann-Morgenstern axioms (see \cite{27}).   \[\]

A complex game in normal form (or strategic form) is an ordered triple $G_{\mathcal{C}}=(N,(S_i)_{i\in N},(h_i)_{i\in N})$, in which:\\\begin{itemize}
    \item $N\subseteq\mathbb{N}$ is the set of players.
    \item $S_i\subseteq\mathbb{C}^{m_i}$ is the set of complex strategies-complex vectors of player $i$, $\forall i\in N$.
    \item $h_i:\displaystyle\prod_{i\in N}S_i\rightarrow\mathbb{R}$ is the payoff (utility) function of player $i$, $\forall i\in N$.
\end{itemize}

Similarly to \cite{19,20} we deal with complex games where $N=2$ and $h_{i}=Re(\cdot,\cdot)$. Before we define pure and mixed strategies and seperate each definition from the other, we restate Theorem 3.1:\\ \\
(Theorem 3.1) \textit{
Let $S_\alpha^m:=\{z\in\mathbb{C}^m:$ $|argz|\leqq\alpha$ and $\sum_{i=1}^{m}z_i=1\}$, where $\alpha\in\mathbb{R}^m$ is of the form $\alpha=\alpha_0e$ for some $\alpha_0\in(0,\frac{\pi}{2})$ and $e=(1,1,...,1)^T$. Then $S_\alpha^m$ is convex and $\mathcal{A}(S^m_\alpha)=\{d^i:i=1,2,...,m^2\}$ where:
\begin{equation*}
    d^i=\left\{
\begin{array}{ll}
      e^i & i=1,2,...,m \\
      \eta^i & i=m+1,m+2,...,m^2\\
\end{array} 
\right. 
\end{equation*}
\[\]
where $e^i=(0,...,0,1,0,...,0)^T$ $(e^i_i=1, e^i_j=0$ $\forall j\neq i)$ and $\eta^i$ are the $m^2-m$ complex vectors defined by all possible placements of elements of the sets $\{\frac{1}{2}+bi\}$ , $\{\frac{1}{2}-bi\}$ and $\{0,0,...,0\}$ (m-2 times) over an m-vector, where $\frac{1}{2}\pm bi$ are such that $|arg(\frac{1}{2}\pm bi)|=\alpha_0$.
    }
\\

 We now give the definitions for pure and mixed complex strategies.
\begin{definition}(Pure complex strategies)
The set $T^{m_i}_{\alpha_i}:=\mathcal{A}(S^{m_i}_{\alpha_i})=\{d^k:k=1,2,...,m_i^2\}$ is the set of pure complex strategies of player $i$. The real vector $\alpha_i:=\alpha_i^0e$ for some real $0<\alpha_i^0<\frac{\pi}{2}$ is called the strategy argument of the set of pure complex strategies of player $i$. 
\end{definition}
\begin{definition}(Mixed complex strategies)
The set $S^{m_i}_{\alpha_i}=\{s_i\in\mathbb{C}^{m_i}:$ $|args_i|\leqq\alpha_i$, $\sum_{k=1}^{m_i}s_{i_k}=1\}$ is the set of mixed complex strategies of player $i$. The real vector $\alpha_i:=\alpha_i^0e$ for some real $0<\alpha_i^0<\frac{\pi}{2}$ is called the strategy argument of the set of mixed complex strategies of player $i$.\\
\end{definition}
From now on we write ``(pure/mixed) strategy" instead of ``(pure/mixed) complex strategy", for simplicity. We also write (without loss of generality) $\alpha_i$ for both the real vector and the real number in $(0,\frac{\pi}{2})$ of the strategy argument; the differentiation between the two cases will be pointed out if needed. Lastly, since we only deal with matrix games, (without loss of generality) we write $m_1:=m$, $m_2:=n$ for some $m,n\in\mathbb{N}$ for the dimensions of the strategy sets of players $I$ and $II$ respectively, and we write $\alpha_1:=\alpha$, $\alpha_2:=\beta$ for the respective strategy arguments.
\[\]

\begin{definition}\label{def:4.13}
Let $G_{\mathcal{C}}(A)$ be a two-player zero-sum complex game as defined above. Then we call maxmin value or security level of player $I$ in pure strategies the value: 
\begin{equation*}
    \underline{h}:=\max\limits_{i\in\mathcal{I}}\min\limits_{j\in\mathcal{J}}Re((d^i)^*Ad^j)
\end{equation*}
Correspondingly, we call maxmin value or security level of player $II$ in pure strategies the value:
\begin{equation*}
    \overline{h}:=\max\limits_{i\in\mathcal{I}}\min\limits_{j\in\mathcal{J}}(-Re((d^i)^*Ad^j))=\min\limits_{j\in\mathcal{J}}\max\limits_{i\in\mathcal{I}}Re((d^i)^*Ad^j)
\end{equation*}
where $\mathcal{I}:=\{1,2,...,m^2\}$ and $\mathcal{J}:=\{1,2,...,n^2\}$\\
\end{definition}

We readily obtain the following results, which are embeded in the already known real--space theory:
\begin{proposition}\label{pr:4.14}
For every two-player zero-sum complex game $G_{\mathcal{C}}(A)$ the following holds: $\underline{h}\leq\overline{h}$
\end{proposition}

\begin{proposition}\label{pr:4.15}
Let $G_\mathcal{C}(A)$ be a two-player zero-sum complex game. Then the strategy $(d^{i_0},d^{j_0})$ is a complex Nash equilibrium in pure strategies if and only if
\begin{equation*}
    Re((d^i)^*Ad^{j_0})\leq Re((d^{i_0})^*Ad^{j_0})\leq Re((d^{i_0})^*Ad^{j})\;\;\forall i\in\mathcal{I}\;\;\forall j\in\mathcal{J}
\end{equation*}
\end{proposition}

\begin{corollary}\label{cor:4.16}
Let $G_\mathcal{C}(A)$ be a two-player zero-sum complex game. If $(d^{i_0},d^{j_0})$ and $(d^{i_1},d^{j_1})$ are complex Nash equilibria then $Re((d^{i_0})^*Ad^{j_0})=Re((d^{i_1})^*Ad^{j_1})$.
\end{corollary}

Now let $G_\mathcal{C}(A)$ be a two-player zero-sum complex game. If $\underline{h}=\overline{h}$, then this common value is the value of the game in pure strategies and is denoted by $v$ or $v_A$. By Corollary \ref{cor:4.16} we obtain that the value of a complex game in pure strategies (if that exists) is fixed for all complex Nash equilibria.\\

\begin{proposition}
In a two-player zero-sum complex game the security levels of players $I$ and $II$ in pure strategies are equal if and only if there exists a complex Nash equilibrium $(d^{i_0},d^{j_0})$ in pure strategies with $v=Re((d^{i_0})^*Ad^{j_0})$.
\end{proposition}
\begin{proof}
$``\Rightarrow"$: Let  $\underline{h}=Re((d^{i_0})^*Ad^{j_1})$ for some $i_0\in\mathcal{I}$, $j_1\in\mathcal{J}$ and $\overline{h}=Re((d^{i_1})^*Ad^{j_0})$ for some $i_1\in\mathcal{I}$, $j_0\in\mathcal{J}$. Then, by Definition \ref{def:4.13}
\begin{equation*}
    \underline{h}=Re((d^{i_0})^*Ad^{j_1})\leq Re((d^{i_0})^*Ad^{j_0})\leq Re((d^{i_1})^*Ad^{j_0})=\overline{h}
\end{equation*}
Since $\underline{h}=\overline{h}=v$, only the equalities hold and therefore by Proposition \ref{pr:4.15} $v=Re((d^{i_0})^*Ad^{j_0})$ and $(d^{i_0},d^{j_0})$ is a complex Nash equilibrium.\\
$``\Leftarrow"$: Let $(d^{i_0},d^{j_0})$ be a complex Nash equilibrium. Then by Proposition \ref{pr:4.15}
\begin{equation*}
    \min\limits_{j\in\mathcal{J}}Re((d^{i_0})^*Ad^j)=Re((d^{i_0})^*Ad^{j_0}),\;\;\max\limits_{i\in\mathcal{I}}Re((d^i)^*Ad^{j_0})=Re((d^{i_0})^*Ad^{j_0})\Rightarrow\underline{h}\geq\overline{h}
\end{equation*}
By Proposition \ref{pr:4.14}\; $\underline{h}\leq\overline{h}$ holds, hence $\underline{h}=\overline{h}=v$.\\
\end{proof}
\par One may notice that a complex Nash equilibrium in pure strategies of a two-player zero-sum complex game does not always exist. Indeed, one can simply take a real payoff matrix with no saddle points and no equilibria that consist of vectors of the form $(0,...,0,\frac{1}{2},0,...,0,\frac{1}{2},0,...,0)^T$, e.g. $A:=$ $
\begin{pmatrix}
0 & 2\\
1 & 0
\end{pmatrix}$.
Similarly to the real case, the ``gap" between the security levels $\underline{h}$ and $\overline{h}$ becomes zero, and a complex Nash equilibrium eventually exists, when we choose to work with the mixed extension of the complex game: The two-player zero-sum complex game $\widetilde{G}_{\mathcal{C}}(A)=(\{I,II\},(S^{m}_\alpha,S^{n}_\beta),(\widetilde{h}_I,\widetilde{h}_{II}))$ is the mixed extension of the two-player zero-sum complex game $G_{\mathcal{C}}(A)=(\{I,II\},(T^{m}_\alpha,T^{n}_\beta),(h_I,h_{II}))$, where the payoff functions $\widetilde{h}_I$ and $\widetilde{h}_{II}$ are defined by
\begin{equation*}
    \widetilde{h}_I(z,w):=Re(z^*Aw),\;\;\widetilde{h}_{II}(z,w):=-Re(z^*Aw)\;\;\forall z\in S^{m}_\alpha\;\;\forall w\in S^{n}_\beta
\end{equation*}
\begin{definition}
Let $\widetilde{G}_{\mathcal{C}}(A)$ be a two-player zero-sum complex game as defined above. Then we call maxmin value or security level of player $I$ in mixed strategies the value: 
\begin{equation*}
    \underline{v}:=\max\limits_{z\in S^{m}_\alpha}\min\limits_{w\in S^{n}_\beta}Re(z^*Aw)
\end{equation*}
Correspondingly, we call maxmin value or security level of player $II$ in mixed strategies the value:
\begin{equation*}
    \overline{v}:=\max\limits_{z\in S^{m}_\alpha}\min\limits_{w\in S^{n}_\beta}(-Re(z^*Aw))=\min\limits_{w\in S^{n}_\beta}\max\limits_{z\in S^{m}_\alpha}Re(z^*Aw)
\end{equation*}
The security level of each player is well defined since the function $Re(z^*Aw)$ is continuous $\forall z\in S^{m}_\alpha$ and $\forall w\in S^{n}_\beta$ (and therefore its max/min is also continuous in the same sets) and the sets $S^{m}_\alpha$, $S^{n}_\beta$ are compact (see Proposition \ref{pr:2.11}).
\end{definition}

From now on we only deal with the mixed extension of a complex game, and therefore (without loss of generality) we write $G_\mathcal{C}(A)$ instead of $\widetilde{G}_\mathcal{C}(A)$ when referring to a two-player zero-sum complex game. Correspondingly, we write $h$ instead of $\widetilde{h}$ for the payoff function. The differentiation between the pure form and the mixed extension of a two-player zero-sum complex game will be pointed out, if necessary.

\begin{proposition}\label{pr:4.25}
For every two-player zero-sum complex game $G_{\mathcal{C}}(A)$ the following holds: $\underline{v}\leq\overline{v}$
\end{proposition}

\begin{proposition}\label{pr:4.28}
Let $G_\mathcal{C}(A)$ be a two-player zero-sum complex game. Then the strategy $(z^0,w^0)$ is a complex Nash equilibrium if and only if
\begin{equation*}
    Re(z^*Aw^0)\leq Re((z^0)^*Aw^0)\leq Re((z^0)^*Aw)\;\;\forall z\in S^{m}_\alpha\;\;\forall w\in S^{n}_\beta
\end{equation*}
\end{proposition}

\begin{corollary}\label{cor:4.29}
Let $G_\mathcal{C}(A)$ be a two-player zero-sum complex game. If $(z^0,w^0)$ and $(z^1,w^1)$ are complex Nash equilibria then $Re((z^0)^*Aw^0)=Re((z^1)^*Aw^1)$.\\
\end{corollary}

Let $G_\mathcal{C}(A)$ be a two-player zero-sum complex game. If $\underline{v}=\overline{v}$ then this common value is the value of the complex game in mixed strategies and is denoted by $v$ or $v_A$. By Corollary \ref{cor:4.29} we obtain that the value of a complex game in mixed strategies (if that exists) is fixed for all complex Nash equilibria.\\ \par
The following theorem is the main result of this section, as it shows that with the use of mixed strategies the players can improve the value of their security level in pure strategies. We will first need a lemma:
\begin{lemma}\label{lem:4.26}
Let $G_{\mathcal{C}}(A)$ be a two-player zero-sum complex game. Then
\begin{equation*}
    \underline{v}=\max\limits_{z\in S^{m}_\alpha}\min\limits_{j\in\mathcal{J}}Re(z^*Ad^j),\;\;\overline{v}=\min\limits_{w\in S^{n}_\beta}\max\limits_{i\in\mathcal{I}}Re((d^i)^*Aw)
\end{equation*}
\end{lemma}
\begin{proof}
We only prove the first (the second equation is similarly proved). The function $g(w):=Re(z^*Aw)$ is linear and continuous in the convex polytope $S^{n}_\beta$ $\;\forall z\in S^{m}_\alpha$, therefore by Corollary \ref{cor:pi} its minimum value (which exists) is taken on an extreme point. That is:
\begin{equation*}
    \min\limits_{w\in S^{n}_\beta}Re(z^*Aw)=\min\limits_{w\in\mathcal{A}(S^{n}_\beta)}Re(z^*Aw)=\min\limits_{j\in\mathcal{J}}Re(z^*Ad^j)\;\;\forall z\in S^{m}_\alpha
\end{equation*}
\begin{equation*}
    \Rightarrow\max\limits_{z\in S^{m}_\alpha}\min\limits_{w\in S^{n}_\beta}Re(z^*Aw)=\max\limits_{z\in S^{m}_\alpha}\min\limits_{j\in\mathcal{J}}Re(z^*Ad^j)
\end{equation*}
\end{proof}

\begin{theorem}
Let $G_{\mathcal{C}}(A)$ be a two-player zero-sum complex game. Then for the security levels of players $I$, $II$ in pure and mixed strategies the following holds:
\begin{equation*}
    \underline{h}\leq\underline{v}\leq\overline{v}\leq\overline{h}
\end{equation*}
\end{theorem}
\begin{proof}
We only prove the first inequality (the second is already known by Proposition \ref{pr:4.25} and the third is proved similarly to the first one). We have $\underline{h}=\max\limits_{i\in\mathcal{I}}\min\limits_{j\in\mathcal{J}}Re((d^i)^*Ad^j)=\max\limits_{z\in\mathcal{A}(S^{m}_\alpha)}\min\limits_{j\in\mathcal{J}}Re(z^*Ad^j)$. Now $\mathcal{A}(S^{m}_\alpha)\subset S^{m}_\alpha$ $\Rightarrow$ $\underline{h}\leq\max\limits_{z\in S^{m}_\alpha}\min\limits_{j\in\mathcal{J}}Re(z^*Ad^j)$. By the previous Lemma we obtain $\underline{h}\leq\underline{v}$
\end{proof}

\[\]
\begin{proposition}\label{th:4.33}
In a two-player zero-sum complex game the security levels of players $I$ and $II$ in mixed strategies are equal if and only if there exists a complex Nash equilibrium $(z^0,w^0)$ in mixed strategies with $v=Re((z^0)^*Aw^0)$.
\end{proposition}
\begin{proof}
$``\Rightarrow"$: Let $\underline{v}=\overline{v}=v$. The functions $\min\limits_{w\in S^{n}_\beta}Re(z^*Aw)$, $\max\limits_{z\in S^{m}_\alpha}Re(z^*Aw)$ are continuous in the compact sets $S^{m}_\alpha$ and $S^{n}_\beta$ respectively, therefore (Weierstrass) they have a maximum and minimum value respectively. Now let $z^0\in S^{m}_\alpha$ and $w^0\in S^{n}_\beta:$
\begin{equation*}
    \underline{v}=\max\limits_{z\in S^{m}_\alpha}\min\limits_{w\in S^{n}_\beta}Re(z^*Aw)=\min\limits_{w\in S^{n}_\beta}Re((z^0)^*Aw)
\end{equation*}
\begin{equation*}
    \overline{v}=\min\limits_{w\in S^{n}_\beta}\max\limits_{z\in S^{m}_\alpha}Re(z^*Aw)=\max\limits_{z\in S^{m}_\alpha}Re(z^*Aw^0)
\end{equation*}
Then $Re((z^0)^*Aw^0)\geq\min\limits_{w\in S^{n}_\beta}Re((z^0)^*Aw)=\underline{v}$ and $Re((z^0)^*Aw^0)\leq\max\limits_{z\in S^{m}_\alpha}Re(z^*Aw^0)=\overline{v}$ $\Rightarrow$ $Re((z^0)^*Aw^0)=\underline{v}=\overline{v}=v$ and $Re(z^*Aw^0)\leq Re((z^0)^*Aw^0)\leq Re((z^0)^*Aw)$ $\;\forall z\in S^{m}_\alpha\;\;\forall w\in S^{n}_\beta$. By Proposition \ref{pr:4.28} $(z^0,w^0)$ is a complex Nash equilibrium.\\
$``\Leftarrow"$: Let $(z^0,w^0)$ be a solution of the complex game. Then
\begin{equation*}
    Re((z^0)^*Aw^0)=\min\limits_{w\in S^{n}_\beta}Re((z^0)^*Aw)=\max\limits_{z\in S^{m}_\alpha}Re(z^*Aw^0)\Rightarrow\underline{v}\geq Re((z^0)^*Aw^0)\geq\overline{v}
\end{equation*}
Now by Proposition \ref{pr:4.25}\; $\underline{v}\leq\overline{v}$ holds, hence $\underline{v}=\overline{v}=v=Re((z^0)^*Aw^0)$.
\end{proof}

\[\]\[\]
\section{The Minimax Theorem}
\;\;\;A proof of the Minimax Theorem in complex space is originally given by Levinson N. \cite{11}, in the more general case where the strategy arguments are not of the form $\alpha e$ for some real $\alpha$. Two different and much simpler proofs are later given by Mond B. and Murray G.J. in their initial work regarding the extension of Game Theory, as seen in \cite{19,20}. In this section we give a completely different proof of the Minimax Theorem, which is constructive and indicates a way of calculating the solution of a two-player zero-sum complex game through the solution of specific complex linear programming problems, similar to the real case. The idea behind the proof follows the one of Dantzig \cite{6}. \[\]

\begin{lemma}\label{lem:5.2}
Let $G_\mathcal{C}(A)$ be a two-player zero-sum complex game. The strategy $(z^0,w^0)$ is a complex Nash equilibrium of $G_\mathcal{C}(A)$ if and only if it is a complex Nash equilibrium of $G_\mathcal{C}(B)$, for every complex payoff matrix $B:=kA+cE$, where $k>0$, $c\in\mathbb{R}$ and $E=(e_{ij})_{i=1,...,m,\;j=1,...,n}$ with $e_{ij}=1$ $\forall i,j$. Moreover, the following holds: $v_B=kv_A+c$ where $v_A$ is the complex game value of $G_\mathcal{C}(A)$ and $v_B$ is the complex game value of $G_\mathcal{C}(B)$. 
\end{lemma}
\begin{proof}
We have $z^*Ew=1$ $\forall z\in S^{m}_\alpha$ $\forall w\in S^{n}_\beta$, since $\displaystyle\sum_{i=1}^{m}z_i=\displaystyle\sum_{j=1}^{n}w_j=1$. Now let $(z^0,w^0)$ be a solution of $G_\mathcal{C}(A)$. Then:
\begin{equation*}
    Re(z^*Aw^0)\leq Re((z^0)^*Aw^0)\leq Re((z^0)^*Aw)\;\;\forall z\in S^{m}_\alpha\;\;\forall w\in S^{n}_\beta
\end{equation*}
\begin{equation*}
    \Leftrightarrow kRe(z^*Aw^0)+c\leq kRe((z^0)^*Aw^0)+c\leq kRe((z^0)^*Aw)+c\;\;\forall z\in S^{m}_\alpha\;\;\forall w\in S^{n}_\beta
\end{equation*}
\begin{equation*}
    \Leftrightarrow kRe(z^*Aw^0)+cRe(z^*Ew^0)\leq kRe((z^0)^*Aw^0)+cRe((z^0)^*Ew^0)\leq kRe((z^0)^*Aw)+cRe((z^0)*Ew)
\end{equation*}
\begin{equation*}
    \Leftrightarrow Re(z^*(kA+cE)w^0)\leq Re((z^0)^*(kA+cE)w^0)\leq Re((z^0)^*(kA+cE)w)\;\;\forall z\in S^{m}_\alpha\;\;\forall w\in S^{n}_\beta
\end{equation*}
\begin{equation*}
    \Leftrightarrow Re(z^*Bw^0)\leq Re((z^0)^*Bw^0)\leq Re((z^0)^*Bw)\;\;\forall z\in S^{m}_\alpha\;\;\forall w\in S^{n}_\beta
\end{equation*}
if and only if $(z^0,w^0)$ is a solution of $G_\mathcal{C}(B)$ by Proposition \ref{pr:4.28}. We then have:
\begin{equation*}
    v_B=Re((z^0)^*Bw^0)=Re((z^0)^*(kA+cE)w^0)=kRe((z^0)^*Aw^0)+c=kv_A+c.
\end{equation*}
\end{proof}
The Minimax Theorem follows:
\[\]
\begin{theorem}(Minimax Theorem)
For every two-player zero-sum complex game the following holds: $\underline{v}=\overline{v}$. That is, every two-player zero-sum complex game has a Nash equilibrium in mixed complex strategies.\\
\end{theorem}
\begin{proof}
Because of Lemma \ref{lem:5.2}, it suffices to show that $\underline{v}=\overline{v}$ for every two-player zero-sum complex game with $\underline{h}>0$. Indeed, let $G_{\mathcal{C}}(A)$ be a two-player zero-sum complex game with $Re((d^{i_0})^*Ad^{j_0})=\underline{h}<0$ for some $i_0\in\mathcal{I}$, $j_0\in\mathcal{J}$. Then there exists $c>0$ in $\mathbb{R}$ such that 
\begin{equation*}
    \max\limits_{i\in\mathcal{I}}\min\limits_{j\in\mathcal{J}}Re((d^{i})^*(A+cE)d^{j})=Re((d^{i_0})^*Ad^{j_0})+c=\underline{h}+c>0
\end{equation*}
The two-player zero-sum complex game $G_\mathcal{C}(B)$ with complex payoff matrix $B:=A+cE$ has $\underline{h}_B>0$. If $(z^0,w^0)$ is a solution of $G_\mathcal{C}(B)$ (we shall show that such solution always exists), then by Lemma \ref{lem:5.2} $(z^0,w^0)$ is also a solution of $G_\mathcal{C}(A)$.\\ \\
Therefore, let $G_\mathcal{C}(A)$ be a two-player zero-sum complex game with $\underline{h}>0$ and strategy arguments $0<\alpha:=\alpha_0e,\;\beta:=\beta_0e<\frac{\pi}{2}e$. Consider the following linear programming problem:
\begin{center}
\begin{equation*}{(P1)\;\;\;\;\;\;\;\;}
\begin{array}{ll}
    max\;\; \xi\\
    s.t.\;\;A^*z-\xi e\in (S^{n}_{\beta})^*\\
    z\in S^{m}_{\alpha} \\
    \xi\geq0
    \end{array}
    \end{equation*}
\end{center}
Then $(P1)$ is equivalent to the problem:
\begin{center}
\begin{equation*}{(P2)\;\;\;\;\;\;\;\;}
\begin{array}{ll}
    max\;\; \xi\\
    s.t.\;\;Re(A^*z,w)\geq \xi Re(e,w),\;\;\forall w\in S^{n}_\beta\\
    \lvert argz\rvert\leqq\alpha\\ 
    \displaystyle\sum_{i=1}^{m}z_i=1\\
    \xi\geq0
    \end{array}
    \end{equation*}
\end{center}

Note that the set $S^{n}_\beta$ is equivalent to the set $\mathcal{T}:=\{w\in\mathbb{C}^{n}:Re(e,w)=1,Im(e,w)=0,\lvert argw\rvert\leqq\beta\}$, therefore $(P2)$ is equivalent to the problem:
\begin{center}
\begin{equation*}{(P3)\;\;\;\;\;\;\;\;}
\begin{array}{ll}
    max\;\; \xi\\
    s.t.\;\;Re(A^*z,w)\geq \xi,\;\;\forall w\in S^{n}_\beta\\
    \lvert argz\rvert\leqq\alpha\\ 
    \displaystyle\sum_{i=1}^{m}z_i=1\\
    \xi\geq0
    \end{array}
    \end{equation*}
\end{center}

We now see that the maximum value of $\xi$ is the maximum value of the function $f(z):=\min\limits_{w\in S^{n}_{\beta}}Re(z^*Aw)$, that is $\xi=\underline{v}$ by definition.\par
We can assume that $\xi>0$: It suffices that $\underline{v}>0$
which is true by our initial hypothesis. Therefore, the problem $(P3)$ is equivalent to the problem:
\begin{center}
\begin{equation*}{(P4)\;\;\;\;\;\;\;\;}
\begin{array}{ll}
    max\;\; \xi\\
    s.t.\;Re(\displaystyle\frac{z}{\xi}^*Aw)\geq 1,\;\;\forall w\in S^{n}_\beta\\
 \lvert arg\displaystyle\frac{z}{\xi}\rvert\leqq\alpha\\ 
    \displaystyle\sum_{i=1}^{m}Re(\frac{z_i}{\xi})=\frac{1}{\xi}\\
    \displaystyle\sum_{i=1}^{m}Im(\frac{z_i}{\xi})=0\\
    \xi>0
    \end{array}
    \end{equation*}
\end{center}
which in turn is equivalent to:
\begin{center}
\begin{equation*}{(P5)\;\;\;\;\;\;\;\;}
\begin{array}{ll}
    min\;\;Re(e^*z') \\
    s.t.\;\;Re(z'^*Aw)\geq 1,\;\;\forall w\in S^{n}_{\beta}\\
     \lvert argz'\rvert\leqq\alpha\\
    \displaystyle\sum_{i=1}^{m}Im(z')=0
    \end{array}
    \end{equation*}
\end{center}
with optimal feasible solution (we show that such solution exists) $\displaystyle\frac{1}{\xi}=\displaystyle\frac{1}{\underline{v}}$, where $z':=\displaystyle\frac{z}{\xi}$.\\ 
Now consider the following problem:

\begin{center}
\begin{equation*}{(P6)\;\;\;\;\;\;\;\;}
\begin{array}{ll}
    min\;\;Re(e^*z') \\
    s.t.\;\;Re(z'^*Aw)\geq 1,\;\;\forall w\in \hat{\mathcal{T}}\\
     \lvert argz'\rvert\leqq\alpha\\
    \displaystyle\sum_{i=1}^{m}Im(z')=0
    \end{array}
    \end{equation*}
\end{center}
where $\hat{\mathcal{T}}:=\{w\in\mathbb{C}^{n}:\lvert argw\rvert\leqq\beta,\;\sum_{j=1}^{n}Im(w)=0\}\supset S^{n}_{\beta}$. If the problem $(P6)$ has optimal feasible solution, say $\theta$, then $\theta\geq\frac{1}{\underline{v}}$ (in fact they are equal). \\

Consider the (dual of $(P1)$) linear programming problem:
\begin{center}
\begin{equation*}{(D1)\;\;\;\;\;\;\;\;}
\begin{array}{ll}
    min\;\; \zeta\\
    s.t.\;\;-Aw+\zeta e \in (S^{m}_\alpha)^*\\
    w\in S^{n}_\beta\\
    \zeta\geq0
    \end{array}
    \end{equation*}
\end{center}
Then $(D1)$ is equivalent to:
\begin{center}
\begin{equation*}{(D2)\;\;\;\;\;\;\;\;}
\begin{array}{ll}
    min\;\; \zeta\\
    s.t.\;\;Re(z^*Aw)\leq \zeta Re(z,e),\;\;\forall z\in S^{m}_\alpha\\
     \lvert argw\rvert\leqq\beta\\ 
    \displaystyle\sum_{j=1}^{n}w_j=1\\
    \zeta\geq0
    \end{array}
    \end{equation*}
\end{center}
Note that the set $S^{m}_\alpha$ is equivalent to the set $\mathcal{S}:=\{z\in\mathbb{C}^{m}:Re(e,z)=1,Im(e,z)=0,\lvert argz\rvert\leqq\alpha\}$, therefore $(D2)$ is equivalent to the problem:
\begin{center}
\begin{equation*}{(D3)\;\;\;\;\;\;\;\;}
\begin{array}{ll}
    min\;\; \zeta\\
    s.t.\;\;Re(z^*Aw)\leq \zeta,\;\;\forall z\in S^{m}_\alpha\\
     \lvert argw\rvert\leqq\beta\\ 
    \displaystyle\sum_{j=1}^{n}w_j=1\\
    \zeta\geq0
    \end{array}
    \end{equation*}
\end{center}

The minimum value of $\zeta$ is the minimum of the function $g(w):=\max\limits_{z\in S^{m}_{\alpha}}Re(z^*Aw)$, hence $\zeta=\overline{v}$ by definition. Now $\underline{h}>0\Rightarrow\zeta\geq\underline{h}>0$. Therefore the problem $(D2)$ is equivalent to the problem:
\begin{center}
\begin{equation*}{(D4)\;\;\;\;\;\;\;\;}
\begin{array}{ll}
    min\;\; \zeta\\
    s.t.\;\;Re(z^*A\displaystyle\frac{w}{\zeta})\leq 1,\;\;\forall z\in S^{m}_\alpha\\
    \lvert arg\displaystyle\frac{w}{\zeta}\rvert\leqq\beta\\ 
    \displaystyle\sum_{j=1}^{n}Re(\frac{w_j}{\zeta})=\frac{1}{\zeta}\\
    \displaystyle\sum_{j=1}^{n}Im(\frac{w_j}{\zeta})=0\\
    \zeta>0
    \end{array}
    \end{equation*}
\end{center}
which in turn is equivalent to:
\begin{center}
\begin{equation*}{(D5)\;\;\;\;\;\;\;\;}
\begin{array}{ll}
    max\;\;Re(e^*w') \\
    s.t.\;\;Re(z^*Aw')\leq 1,\;\;\forall z\in S^{m}_\alpha\\
     \lvert argw'\rvert\leqq\beta\\
    \displaystyle\sum_{j=1}^{n}Im(w')=0
    \end{array}
    \end{equation*}
\end{center}
with optimal feasible solution (we show that such solution exists) $\displaystyle\frac{1}{\zeta}=\displaystyle\frac{1}{\overline{v}}$, where $w':=\displaystyle\frac{w}{\zeta}$. Now consider the following problem:
\begin{center}
\begin{equation*}{(D6)\;\;\;\;\;\;\;\;}
\begin{array}{ll}
    max\;\;Re(e^*w') \\
    s.t.\;\;Re(z^*Aw')\leq 1,\;\;\forall z\in \hat{\mathcal{S}}\\
     \lvert argw'\rvert\leqq\beta\\
    \displaystyle\sum_{j=1}^{n}Im(w')=0
    \end{array}
    \end{equation*}
\end{center}
where $\hat{\mathcal{S}}:=\{z\in\mathcal{C}^{m}:\lvert argz \rvert\leqq\alpha,\sum_{i=1}^{m}Im(z)=0\}\supset S^{m}_\alpha$. If the problem $(D6)$ has optimal feasible solution, say $\eta$, then $\eta\leq\frac{1}{\overline{v}}$ (in fact they are equal).\\ \par
Now $(P6)$ and $(D6)$ are equivalent to the problems $(P^*)$ and $(D^*)$ respectively, as these are defined by:
\begin{center}
\begin{equation*}{(P^*)\;\;\;\;\;\;\;\;}
\begin{array}{ll}
    min\;\;Re(e^*z') \\
    s.t.\;\;Az'-e \in(\hat{\mathcal{T}})^*\\
     z'\in\hat{\mathcal{S}}
    \end{array}
    \end{equation*}
\end{center}

\begin{center}
\begin{equation*}{(D^*)\;\;\;\;\;\;\;\;}
\begin{array}{ll}
    max\;\;Re(e^*w') \\
    s.t.\;\;-A^*w'+e\in(\hat{\mathcal{S}})^*\\
     w'\in\hat{\mathcal{T}}
    \end{array}
    \end{equation*}
\end{center}

\[\]
The sets $\hat{\mathcal{S}},\hat{\mathcal{T}}$ are closed convex cones and the problems $(P^*),\;(D^*)$ are dual and feasible (if such solutions are “difficult” for the reader to find, one can consider the matrix $A+cE$ with $c>0$ big enough such that all elements of $A+cE$ have a really small argument, and then find suitable $z',\;w'$ that have all imaginary parts equal to zero). By using either the known strong duality theory in conic finite--linear programming (e.g. see \cite{8, 25}) or previous related results in complex linear programming (see \cite{18}) we obtain $\eta=\theta$, and therefore $\underline{v}\geq\overline{v}$. By Proposition \ref{pr:4.25} we finally obtain $\underline{v}=\overline{v}=v_A$. Then by Proposition \ref{th:4.33} there exists a complex Nash equilibrium in mixed strategies.
\end{proof}
\[\]\[\]
\section{Domination and Elimination}
\;\;\;\;\;\;\;\;\;As in the real case, in a (two-player zero-sum) complex game there may be strategies that are weakly or strictly dominated by other strategies. Due to the more mathematically complicated nature of complex matrices however, the generic process of iterated elimination, as defined and explained in the real case, cannot be applied here. In this section we show that under a specific condition, called the ``condition of elimination", the process can be applied so that the solution of the game remains unchanged. We give two type-different definitions of elimination, which are included in a third, more generic one. This seperation is made for practical reasons and shall be explained and clarified later on through an example.
\[\]
We give the definitions of dominated strategies in the case of two-player zero-sum complex games. We begin with the first type of domination:
\begin{definition}\label{def:6.3}
Let $G_\mathcal{C}(A)$ be a two-player zero-sum complex game. A row $i_0$ of $A$ is weakly dominated if there exists a row $i_1$ of $A$ such that:
\begin{enumerate}
    \item \begin{equation*}
        Re((e^{i_0})^*Ad^j)\leq Re((e^{i_1})^*Ad^j)\;\;\forall j\in\mathcal{J}
    \end{equation*}
    \item there exists an extreme point $d^{j_0}$ for some $j_0\in\mathcal{J}$:
    \begin{equation*}
        Re((e^{i_0})^*Ad^j_0)< Re((e^{i_1})^*Ad^j_0)
    \end{equation*}
\end{enumerate}
In this case we say that row $i_0$ is weakly dominated by row $i_1$ and that row $i_1$ weakly dominates row $i_0$.\\ \\
Correspondingly, a column $j_0$ of $A$ is weakly dominated if there exists a column $j_1$ of $A$ such that:
\begin{enumerate}
    \item \begin{equation*}
        Re((d^i)^*Ae^{j_0})\geq Re((d^i)^*Ae^{j_1})\;\;\forall i\in\mathcal{I}
    \end{equation*}
    \item there exists an extreme point $d^{i_0}$ for some $i_0\in\mathcal{I}$:
    \begin{equation*}
        Re((d^{i_0})^*Ae^{j_0})> Re((d^{i_0})^*Ae^{j_1})
    \end{equation*}
\end{enumerate}
In this case we say that column $j_0$ is weakly dominated by column $j_1$ and that column $j_1$ weakly dominates column $j_0$.
\end{definition}
\begin{definition}\label{def:6.4}
Let $G_\mathcal{C}(A)$ be a two-player zero-sum complex game. A row $i_0$ of $A$ is strictly dominated if there exists a row $i_1$ of $A$ such that:
\begin{equation*}
        Re((e^{i_0})^*Ad^j)< Re((e^{i_1})^*Ad^j)\;\;\forall j\in\mathcal{J}
\end{equation*}

In this case we say that row $i_0$ is strictly dominated by row $i_1$ and that row $i_1$ strictly dominates row $i_0$.\\ \\
Correspondingly, a column $j_0$ of $A$ is strictly dominated if there exists a column $j_1$ of $A$ such that
\begin{equation*}
        Re((d^i)^*Ae^{j_0})> Re((d^i)^*Ae^{j_1})\;\;\forall i\in\mathcal{I}
\end{equation*}
In this case we say that column $j_0$ is strictly dominated by column $j_1$ and that column $j_1$ strictly dominates column $j_0$.
\end{definition}
We now give the second type of domination:
\begin{definition}\label{def:6.5}
Let $G_\mathcal{C}(A)$ be a two-player zero-sum complex game. A row $i_0$ of $A$ is weakly dominated if there exist rows $i_1,i_2$ of $A$ and $\lambda\in(0,1)$ such that:
\begin{enumerate}
    \item \begin{equation*}
        Re((e^{i_0})^*Ad^j)\leq Re((\lambda e^{i_1}+(1-\lambda)e^{i_2})^*Ad^j)\;\;\forall j\in\mathcal{J}
    \end{equation*}
    \item there exists an extreme point $d^{j_0}$ for some $j_0\in\mathcal{J}$:
    \begin{equation*}
        Re((e^{i_0})^*Ad^{j_0})< Re((\lambda e^{i_1}+(1-\lambda)e^{i_2})^*Ad^{j_0})
    \end{equation*}
\end{enumerate}
In this case we say that row $i_0$ is weakly dominated by rows $i_1,i_2$ and that rows $i_1,i_2$ weakly dominate row $i_0$.\\ \\
Correspondingly, a column $j_0$ of $A$ is weakly dominated if there exist columns $j_1,j_2$ of $A$ and $\lambda\in(0,1)$ such that:
\begin{enumerate}
    \item \begin{equation*}
        Re((d^i)^*Ae^{j_0})\geq Re((d^i)^*A(\lambda e^{j_1}+(1-\lambda)e^{j_2}))\;\;\forall i\in\mathcal{I}
    \end{equation*}
    \item there exists an extreme point $d^{i_0}$ for some $i_0\in\mathcal{I}$:
    \begin{equation*}
        Re((d^{i_0})^*Ae^{j_0})> Re((d^{i_0})^*A(\lambda e^{j_1}+(1-\lambda)e^{j_2}))
    \end{equation*}
\end{enumerate}
In this case we say that column $j_0$ is weakly dominated by columns $j_1,j_2$ and that columns $j_1,j_2$ weakly dominate column $j_0$.
\end{definition}\[\]
\begin{definition}\label{def:6.6}
Let $G_\mathcal{C}(A)$ be a two-player zero-sum complex game. A row $i_0$ of $A$ is strictly dominated if there exist rows $i_1,i_2$ of $A$ and $\lambda\in(0,1)$ such that:
\begin{equation*}
        Re((e^{i_0})^*Ad^j)< Re((\lambda e^{i_1}+(1-\lambda)e^{i_2})^*Ad^j)\;\;\forall j\in\mathcal{J}
\end{equation*}

In this case we say that row $i_0$ is strictly dominated by rows $i_1,i_2$ and that rows $i_1,i_2$ strictly dominate row $i_0$.\\ \\
Correspondingly, a column $j_0$ of $A$ is strictly dominated if there exist columns $j_1,j_2$ of $A$ and $\lambda\in(0,1)$ such that
\begin{equation*}
        Re((d^i)^*Ae^{j_0})> Re((d^i)^*A(\lambda e^{j_1}+(1-\lambda)e^{j_2}))\;\;\forall i\in\mathcal{I}
\end{equation*}
In this case we say that column $j_0$ is strictly dominated by columns $j_1,j_2$ and that columns $j_1,j_2$ strictly dominate column $j_0$.
\end{definition}
\[\]
The following definition is the general definition of domination in the case of two-player zero-sum complex games, and includes both types given above:
\begin{definition}\label{def:6.7}
Let $G_\mathcal{C}(A)$ be a two-player zero-sum complex game. A row $i_0$ of $A$ is weakly dominated if there exists a mixed strategy $z\in S^{m}_\alpha$:
\begin{enumerate}
    \item \begin{equation*}
        Re((e^{i_0})^*Ad^j)\leq Re(z^*Ad^j)\;\;\forall j\in\mathcal{J}
    \end{equation*}
    \item there exists a complex strategy $d^{j_0}$ for some $j_0\in\mathcal{J}:$
    \begin{equation*}
         Re((e^{i_0})^*Ad^{j_0})<Re(z^*Ad^{j_0})
    \end{equation*}
\end{enumerate}
In this case we say that row $i_0$ is weakly dominated by mixed strategy $z$ and that mixed strategy $z$ weakly dominates row $i_0$.\\ \\
Correspondingly, a column $j_0$ is weakly dominated if there exists a mixed strategy $w\in S^{n}_\beta$:
\begin{enumerate}
    \item \begin{equation*}
        Re((d^i)^*Ae^{j_0})\geq Re((d^i)^*Aw)\;\;\forall i\in\mathcal{I}
    \end{equation*}
    \item there exists a complex strategy $d^{i_0}$ for some $i_0\in\mathcal{I}:$
    \begin{equation*}
         Re((d^{i_0})^*Ae^{j_0})>Re((d^{i_0})^*Aw)
    \end{equation*}
\end{enumerate}
In this case we say that column $j_0$ is weakly dominated by mixed strategy $w$ and that mixed strategy $w$ weakly dominates column $j_0$.
\end{definition}
\begin{definition}\label{def:6.8}
Let $G_\mathcal{C}(A)$ be a two-player zero-sum complex game. A row $i_0$ of $A$ is strictly dominated if there exists a mixed strategy $z\in S^{m}_\alpha$:
\begin{equation*}
        Re((e^{i_0})^*Ad^j)<Re(z^*Ad^j)\;\;\forall j\in\mathcal{J}
\end{equation*}
In this case we say that row $i_0$ is strictly dominated by mixed strategy $z$ and that mixed strategy $z$ strictly dominates row $i_0$.\\ \\
Correspondingly, a column $j_0$ is strictly dominated if there exists a mixed strategy $w\in S^{n}_\beta$:
\begin{equation*}
        Re((d^i)^*Ae^{j_0})> Re((d^i)^*Aw)\;\;\forall i\in\mathcal{I}
\end{equation*}
In this case we say that column $j_0$ is strictly dominated by mixed strategy $w$ and that mixed strategy $w$ strictly dominates column $j_0$.\\
\end{definition}
\begin{remark}
If a row $i_0$ is weakly or strictly dominated by mixed strategy $z$ then $\min\limits_{j\in\mathcal{J}}Re((e^{i_0})^*Ad^j)\leq\min\limits_{j\in\mathcal{J}} Re(z^*Ad^j)\leq\underline{v}$, with the first inequality being strict for strict domination. Similarly, if a column $j_0$ is weakly or strictly dominated by mixed strategy $w$ then $\max\limits_{i\in\mathcal{I}}Re((d^i)^*Ae^{j_0})\geq\max\limits_{i\in\mathcal{I}} Re((d^i)^*Aw)\geq\overline{v}$, with the first inequality being strict for strict domination. Therefore, in both cases, a rational player chooses to play mixed strategy $z$ instead of pure strategy $e^{i_0}$ (or respectively $w$ instead of $e^{j_0}$), since it gives him a payoff ``closer" to his/her security level. Note that this preference still holds for $z$ (or $w$ respectively) being a trivial extreme point or a convex combination of trivial extreme points.
\end{remark}\[\]
\;\;\;\;\;\;\;\;\;As we already mentioned in the beginning of this section, the reason why we seperate Definitions \ref{def:6.3} and \ref{def:6.5} from the more general definition \ref{def:6.7} (and \ref{def:6.4}, \ref{def:6.6} from \ref{def:6.8} respectively), even though they are both included in it, is practical: By the process of iterated elimination, rows and columns of payoff matrix $A$ that are weakly or strictly dominated by other strategies may be eliminated, leading to a payoff matrix $A'$ with less rows and columns, and therefore making the search for the solution an easier job, as it will be seen in Section 8. The search for such dominated rows and columns through a finite number of calculations is easier and less time-consuming when done by following the first two types of domination.\par
\;\;\;\;We now prove that a game $G_\mathcal{C}(A')$ with payoff matrix $A'$ which results from elimination of rows and columns of a payoff matrix $A$ has the same solutions with the complex game $G_\mathcal{C}(A)$, as far as the ``condition of elimination" is met. We first need the following simple Lemma:\[\]
\begin{lemma}\label{lem:6.10}
If $Re((e^{i_0})^*Ad^j)\leq Re((e^{i_1})^*Ad^j)\;\;\forall j\in\mathcal{J}$ then $Re((e^{i_0})^*Aw)\leq Re((e^{i_1})^*Aw)\;\;\forall w\in S^{n}_\beta$
\end{lemma}
\begin{proof}
Let $w\in S^{n}_\beta$. Since $S^{n}_\beta$ is a convex polytope there exist $\lambda_1,\lambda_2,...,\lambda_k\in[0,1]$ and extreme points $d^{j_1},d^{j_2},...,d^{j_k}$ for some integer $k$ such that $\displaystyle\sum_{l=1}^{k}\lambda_l=1$ and $w=\displaystyle\sum_{l=1}^{k}\lambda_ld^{j_l}$. Then:
\begin{equation*}
    Re((e^{i_0})^*Aw)=Re((e^{i_0})^*A(\lambda_1d^{j_1}+...+\lambda_kd^{j_k}))=\lambda_1Re((e^{i_0})^*Ad^{j_1})+...+\lambda_kRe((e^{i_0})^*Ad^{j_k})\leq
\end{equation*}
\begin{equation*}
    \leq\lambda_1Re((e^{i_1})^*Ad^{j_1})+...+\lambda_kRe((e^{i_1})^*Ad^{j_k})=Re((e^{i_1})^*Aw)\;\;\;\;\;\;\;\;\;\;
\end{equation*}
\end{proof}
\begin{theorem}
Let $A'$ be a payoff matrix that results from the elimination of row $i_0$ of a payoff matrix $A$, where row $i_0$ is weakly dominated by row $i_1$. If $(z^0:=(z^0_1,z^0_2,...,z^0_{i_0-1},z^0_{i_0+1},...,z^0_m),w^0:=(w^0_1,...,w^0_n))$ is a complex Nash equilibrium of $G_\mathcal{C}(A')$ and if $Im((e^{i_0})^*Aw^0)=Im((e^{i_1})^*Aw^0)$ (``condition of elimination"), then $(z'^0:=(z^0_1,...,z^0_{i_0-1},0,z^0_{i_0+1},...,z^0_m),w^0)$ is a complex Nash equilibrium of $G_\mathcal{C}(A)$ and $v_{A'}=v_A$. 
\end{theorem}
\begin{proof}(adapted from \cite{16})
Without loss of generality we can assume that $i_0$ is the first row and $i_1$ is the second row of $A$. We now prove that the following holds:
\begin{equation}
    Re(z^*Aw^0)\leq Re((z'^0)^*Aw^0)\leq Re((z'^0)^*Aw)\;\;\forall z\in S^{m}_\alpha\;\;\forall w\in S^{n}_\beta
\end{equation}\\
For the right inequality: since $(z^0,w^0)$ is a complex Nash equilibrium of $G_\mathcal{C}(A')$ the following holds:
\begin{equation}
    Re((\hat z^*A'w^0)\leq Re((z^0)^*A'w^0)\leq Re((z^0)^*A'w)\;\;\forall\hat{z}\in S^{m-1}_\alpha\;\;\forall w\in S^{n}_\beta
\end{equation}
One can now see that $(z^0)^*A'=(z'^0)^*A$ and therefore for some $w\in S^{n}_\beta:$
\begin{equation*}
    Re((z^0)^*A'w^0)\leq Re((z^0)^*A'w)\Leftrightarrow Re((z'^0)^*Aw^0)\leq Re((z'^0)^*Aw)
\end{equation*}
Hence by (6.2) the right inequality of (6.1) holds.\\ \\
For the left inequality: We have
\begin{equation*}
    Re((\hat z^*A'w^0)\leq Re((z^0)^*A'w^0)\;\;\forall\hat{z}\in S^{m-1}_\alpha\mbox\;{ and }\; (z^0)^*A'w^0=(z'^0)^*Aw^0\Rightarrow
\end{equation*}
\begin{equation*}
    \Rightarrow Re((z^0)^*A'w^0)=Re((z'^0)^*Aw^0)\geq Re(\hat{z}^*A'w^0)\;\;\forall\hat{z}\in S^{m-1}_\alpha
\end{equation*}
Therefore it suffices to show that:
\begin{equation*}
  \forall z\in S^{m}_\alpha\;\;\exists\hat{z}\in S^{m-1}_\alpha:\;\;Re(\hat{z}^*A'w^0)\geq Re(z^*Aw^0)  
\end{equation*}
Let $z=(z_1,z_2,...,z_m)\in S^{m}_\alpha$ and consider the complex vector $\hat{z}:=(z_1+z_2,z_3,z_4,...,z_m)$. Then $\hat{z}\in S^{m-1}_\alpha:$
\begin{equation*}
    \lvert argz_1\rvert\leq\alpha\;\mbox{ and}\;\;\lvert argz_2\rvert\leq\alpha\Rightarrow\lvert arg(z_1+z_2)\rvert\leq\alpha\;,\mbox{ and}\;\;\displaystyle\sum_{i=1}^{m-1}\hat{z}_i=\displaystyle\sum_{j=1}^{m}z_j=1
\end{equation*}
We now have:
\begin{equation*}
    Re(\hat{z}^*A'w^0)=Re(\overline{z_1+z_2}a_{21}w^0_1+\displaystyle\sum_{i=3}^{m}\overline{z_i}a_{i1}w_1^0)+Re(\overline{z_1+z_2}a_{22}w^0_2+\displaystyle\sum_{i=3}^{m}\overline{z_i}a_{i2}w_2^0)+...+
\end{equation*}
\begin{equation*}
    +Re(\overline{z_1+z_2}a_{2n}w^0_n+\displaystyle\sum_{i=3}^{m}\overline{z_i}a_{in}w_n^0)\;\mbox{ and }
\end{equation*}
\begin{equation*}
    Re(z^*Aw^0)=Re(\displaystyle\sum_{i=1}^{m}\overline{z_i}a_{i1}w^0_1)+Re(\displaystyle\sum_{i=1}^{m}\overline{z_i}a_{i2}w^0_2)+...+Re(\displaystyle\sum_{i=1}^{m}\overline{z_i}a_{in}w^0_n)
\end{equation*}\\ \\
Therefore we just have to prove that $Re((z_1e^1)^*Aw^0)\leq Re((z_1e^2)^*Aw^0)$ (why?), where:
\begin{equation*}
    Re((z_1e^1)^*Aw^0)=Re(z_1)Re((e^1)^*Aw^0)+Im(z_1)Im((e^1)^*Aw^0)
\end{equation*}
\begin{equation*}
    Re((z_1e^2)^*Aw^0)=Re(z_1)Re((e^2)^*Aw^0)+Im(z_1)Im((e^2)^*Aw^0)
\end{equation*}
The inequality $Re((e^1)^*Aw^0)\leq Re((e^2)^*Aw^0)$ holds by Lemma \ref{lem:6.10} since row $1$ is weakly dominated by row $2$. The equality $Im((e^1)^*Aw^0)=Im((e^2)^*Aw^0)$ also holds, since it represents the condition of elimination. Hence $Re((z_1e^1)^*Aw^0)\leq Re((z_1e^2)^*Aw^0)\Rightarrow (z'^0,w^0)$ is a complex Nash equilirium of $G_{\mathcal{C}}(A)$ and
\begin{equation*}
    v_{A'}=Re((z^0)^*A'w^0)=Re((z'^0)^*Aw^0)=v_A
\end{equation*}
\end{proof}
 \par The above theorem is also true for the case where a row $i_0$ is strictly dominated by another row $i_1$. The proof is exactly the same, while the left inequality is always strict. The respective theorem can be similarly proved in the case where a column $j_0$ is weakly (strictly) dominated by column $j_1$: the condition of elimination in this case is $Im((z^0)^*Ae^{j_0})=Im((z^0)^*Ae^{j_1})$ (we leave the details to the reader).\\

A respective theorem is proved in the case where the domination follows the type given in Definition \ref{def:6.5}. In fact:
\begin{lemma}\label{lem:6.13} 
If $Re((e^{i_0})^*Ad^j)\leq Re((\lambda e^{i_1}+(1-\lambda)e^{i_2})^*Ad^j)\;\;\forall j\in\mathcal{J}$ for some $\lambda\in(0,1)$ then $Re((e^{i_0})^*Aw)\leq Re((\lambda e^{i_1}+(1-\lambda)e^{i_2})^*Aw)\;\;\forall w\in S^{n}_\beta$
\end{lemma}
The proof is similar to the one of Lemma \ref{lem:6.10} and is left to the reader as an exercise.\\
\begin{theorem}
Let $A'$ be a payoff matrix that results from the elimination of row $i_0$ of a payoff matrix $A$, where row $i_0$ is weakly dominated by rows $i_1,i_2$ for some $\lambda\in(0,1)$. If $(z^0:=(z^0_1,z^0_2,...,z^0_{i_0-1},z^0_{i_0+1},...,z^0_m),w^0:=(w^0_1,...,w^0_n))$ is a complex Nash equilibrium of $G_\mathcal{C}(A')$ and if $Im((e^{i_0})^*Aw^0)=\lambda Im((e^{i_1})^*Aw^0)+(1-\lambda)Im((e^{i_2})^*Aw^0)$ (condition of elimination), then $(z'^0:=(z^0_1,...,z^0_{i_0-1},0,z^0_{i_0+1},...,z^0_m),w^0)$ is a complex Nash equilibrium of $G_\mathcal{C}(A)$ and $v_{A'}=v_A$. 
\end{theorem}
 \begin{proof}
Without loss of generality we can assume that $i_0$ is the first row, $i_1$ is the second row and $i_2$ is the third row of $A$. We now prove that the following holds:
\begin{equation}
    Re(z^*Aw^0)\leq Re((z'^0)^*Aw^0)\leq Re((z'^0)^*Aw)\;\;\forall z\in S^{m}_\alpha\;\;\forall w\in S^{n}_\beta
\end{equation}\\
For the right inequality: since $(z^0,w^0)$ is a complex Nash equilibrium of $G_\mathcal{C}(A)$ the following holds:
\begin{equation}
    Re((\hat z^*A'w^0)\leq Re((z^0)^*A'w^0)\leq Re((z^0)^*A'w)\;\;\forall\hat{z}\in S^{m-1}_\alpha\;\;\forall w\in S^{n}_\beta
\end{equation}
One can now see that $(z^0)^*A'=(z'^0)^*A$ and therefore for some $w\in S^{n}_\beta:$
\begin{equation*}
    Re((z^0)^*A'w^0)\leq Re((z^0)^*A'w)\Leftrightarrow Re((z'^0)^*Aw^0)\leq Re((z'^0)^*Aw)
\end{equation*}
Hence by (6.4) the right inequality of (6.3) holds.\\ \\
For the left inequality: We have
\begin{equation*}
    Re((\hat z^*A'w^0)\leq Re((z^0)^*A'w^0)\;\;\forall\hat{z}\in S^{m-1}_\alpha\mbox\;{ and }\; (z^0)^*A'w^0=(z'^0)^*Aw^0\Rightarrow
\end{equation*}
\begin{equation*}
    \Rightarrow Re((z^0)^*A'w^0)=Re((z'^0)^*Aw^0)\geq Re(\hat{z}^*A'w^0)\;\;\forall\hat{z}\in S^{m-1}_\alpha
\end{equation*}
Therefore it suffices to show that:
\begin{equation*}
  \forall z\in S^{m}_\alpha\;\;\exists\hat{z}\in S^{m-1}_\alpha:\;\;Re(\hat{z}^*A'w^0)\geq Re(z^*Aw^0)  
\end{equation*}
Let $z=(z_1,z_2,...,z_m)\in S^{m}_\alpha$ and consider the complex vector $\hat{z}:=(\lambda z_1+z_2,(1-\lambda)z_1+z_3,z_4,...,z_m)$. Then $\hat{z}\in S^{m-1}_\alpha:$
\begin{equation*}
    \lvert argz_1\rvert\leq\alpha\;,\lvert argz_2\rvert\leq\alpha\;\mbox{ and}\;\;\lvert argz_3\rvert\leq\alpha\Rightarrow\lvert arg(\lambda z_1+z_2)\rvert\leq\alpha\;,\lvert arg((1-\lambda) z_1+z_3)\rvert\leq\alpha
\end{equation*}
\begin{equation*}
    \mbox{and }\;\displaystyle\sum_{i=1}^{m-1}\hat{z}_i=\displaystyle\sum_{j=1}^{m}z_j=1
\end{equation*}
We now have:
\begin{equation*}
    Re(\hat{z}^*A'w^0)=Re(\overline{\lambda z_1+z_2}a_{21}w^0_1+\overline{(1-\lambda) z_1+z_3}a_{31}w^0_1+\displaystyle\sum_{i=4}^{m}\overline{z_i}a_{i1}w_1^0)+
\end{equation*}
\begin{equation*}
    +Re(\overline{\lambda z_1+z_2}a_{22}w^0_2+\overline{(1-\lambda) z_1+z_3}a_{32}w^0_2+\displaystyle\sum_{i=4}^{m}\overline{z_i}a_{i2}w_2^0)+...\;+
\end{equation*}
\begin{equation*}
    +Re(\overline{\lambda z_1+z_2}a_{2n}w^0_n+\overline{(1-\lambda) z_1+z_3}a_{3n}w^0_n+\displaystyle\sum_{i=4}^{m}\overline{z_i}a_{in}w_n^0)\;\mbox{ and }
\end{equation*}
\begin{equation*}
    Re(z^*Aw^0)=Re(\displaystyle\sum_{i=1}^{m}\overline{z_i}a_{i1}w^0_1)+Re(\displaystyle\sum_{i=1}^{m}\overline{z_i}a_{i2}w^0_2)+...+Re(\displaystyle\sum_{i=1}^{m}\overline{z_i}a_{in}w^0_n)
\end{equation*}\\ \\
Therefore we just have to prove that $Re((z_1e^1)^*Aw^0)\leq \lambda Re((z_1e^2)^*Aw^0)+(1-\lambda)Re((z_1e^3)^*Aw^0)$ (why?), where:
\begin{equation*}
    Re((z_1e^1)^*Aw^0)=Re(z_1)Re((e^1)^*Aw^0)+Im(z_1)Im((e^1)^*Aw^0)
\end{equation*}
\begin{equation*}
    \lambda Re((z_1e^2)^*Aw^0)=Re(z_1)\lambda Re((e^2)^*Aw^0)+Im(z_1)\lambda Im((e^2)^*Aw^0)
\end{equation*}
\begin{equation*}
    (1-\lambda)Re((z_1e^3)^*Aw^0)=Re(z_1)(1-\lambda)Re((e^3)^*Aw^0)+Im(z_1)(1-\lambda)Im((e^3)^*Aw^0)
\end{equation*}\\
The inequality $Re((e^1)^*Aw^0)\leq \lambda Re((e^2)^*Aw^0)+(1-\lambda)Re((e^3)^*Aw^0)$ holds by Lemma \ref{lem:6.13} since row $1$ is weakly dominated by rows $2$ and $3$. The equality $Im((e^1)^*Aw^0)=\lambda Im((e^2)^*Aw^0)+(1-\lambda)Im((e^3)^*Aw^0)$ also holds, since it represents the condition of elimination. Hence $Re((z_1e^1)^*Aw^0)\leq \lambda Re((z_1e^2)^*Aw^0)+(1-\lambda)Re((z_1e^3)^*Aw^0)\Rightarrow (z'^0,w^0)$ is a complex Nash equilirium of $G_{\mathcal{C}}(A)$ and\\
\begin{equation*}
    v_{A'}=Re((z^0)^*A'w^0)=Re((z'^0)^*Aw^0)=v_A
\end{equation*}
\end{proof}
\par The above theorem is also true for the case where a row $i_0$ is strictly dominated by rows $i_1,i_2$ for some $\lambda\in(0,1)$. The proof is exactly the same, while the left inequality is always strict. The respective theorem can be similarly proved in the case where a column $j_0$ is weakly (strictly) dominated by columns $j_1,j_2$ for some $\lambda\in(0,1)$: the condition of elimination in this case is $Im((z^0)^*Ae^{j_0})=\lambda Im((z^0)^*Ae^{j_1})+(1-\lambda)Im((z^0)^*Ae^{j_2})$ (we leave the details to the reader).\\

\par
The process of iterated elimination of rows and columns of a payoff matrix, as it is described in the above through the search of weakly and strictly dominated rows and columns, facilitates the process of finding a game solution. Although the condition of elimination may seem strict and rare (especially the second type), there is a case where it is always met, no matter the number of dominations. Obviously, one can see that in the case of iterated eliminations and the formation of smaller and smaller submatrices, a number of iterated conditions of eliminations must be met, equal to the number of dominations, in order for the solution to remain unchanged.
\[\]\[\]
\section{Equalizing complex strategies}
In this section we define equalizing complex strategies, as they are defined in the real case. We show that by using such strategies one can easily calculate the solution of a two-player zero-sum complex game, if certain constraints are met.\\ \\
\begin{definition}\label{def:7.1}
Let $G_\mathcal{C}(A)$ be a two-player zero-sum complex game. A mixed strategy $z\in S^{m}_\alpha$ of player $I$ is called equalizing if $Re(z^*Ad^j)=c$ $\;\forall j\in\mathcal{J}$ for some $c\in\mathbb{R}$. Correspondingly, a mixed strategy $w\in S^{n}_\beta$ of player $II$ is called equalizing if $Re((d^i)^*Aw)=c$ $\;\forall i\in\mathcal{I}$ for some $c\in\mathbb{R}$. In both cases the real constant $c$ is called equalizing.
\end{definition}
\begin{lemma}\label{lem:7.2}
Let $z^0\in S^{m}_\alpha$ be an equalizing strategy of player $I$ and let $c$ be the equalizing constant. Then:
\begin{equation*}
    Re((z^0)^*Aw)=c\;\;\forall w\in S^{n}_\beta
\end{equation*}
Correspondingly, let $w^0\in S^{n}_\beta$ be an equalizing strategy of player $II$ and let $c$ be the equalizing constant. Then:
\begin{equation*}
    Re(z^*Aw^0)=c\;\;\forall z\in S^{m}_\alpha
\end{equation*}
\end{lemma}
\begin{proof}
We only prove the first relation (the second is similarly proved). Let $z^0$ be an equalizing strategy of player $I$ and $w\in S^{n}_\beta$. There exist $\lambda_1,...,\lambda_k\in[0,1]$ and extreme points $d^{j_1},...,d^{j_k}$ for some integer $k$ such that
\begin{equation*}
    \displaystyle\sum_{i=1}^{k}\lambda_i=1\;\mbox{ and}\;\;\displaystyle\sum_{i=1}^{k}\lambda_id^{j_i}=w
\end{equation*}\\ \\
Then: $Re((z^0)^*Aw)=\lambda_1 Re((z^0)^*Ad^{j_1})+\lambda_2 Re((z^0)^*Ad^{j_2})+...+\lambda_k Re((z^0)^*Ad^{j_k})=\lambda_1 c+\lambda_2 c+...+\lambda_k c=c$
\end{proof}
\[\]
\begin{theorem}\label{th:7.3}
Let $G_\mathcal{C}(A)$ be a two-player zero-sum complex game. If $z^0\in S^{m}_\alpha$ is an equalizing strategy of player $I$ with equalizing constant $c_1$ and $w^0\in S^{n}_\beta$ is an equalizing strategy of player $II$ with equalizing constant $c_2$ then:
\begin{enumerate}
    \item strategies $z^0,w^0$ are optimal, that is $(z^0,w^0)$ is a complex Nash equilibrium of $G_\mathcal{C}(A)$ and
    \item $c_1=c_2=v_A$
\end{enumerate}
\end{theorem}
\begin{proof}
Strategies $z^0$ and $w^0$ are equalizing, therefore by Lemma \ref{lem:7.2} we have:\\
\begin{equation}
    Re((z^0)^*Aw)=c_1\;\;\forall w\in S^{n}_\beta\Rightarrow Re((z^0)^*Aw^0)=c_1
\end{equation}
\begin{equation}
    Re(z^*Aw^0)=c_2\;\;\forall z\in S^{m}_\alpha\Rightarrow Re((z^0)^*Aw^0)=c_2
\end{equation}\\
By (7.1) and (7.2) we obtain $c_1=c_2$ and $Re(z^*Aw^0)=Re((z^0)^*Aw^0)=Re((z^0)^*Aw)$ $\forall z\in S^{m}_\alpha\;\;\forall w\in S^{n}_\beta$. Therefore by Proposition \ref{pr:4.28} $(z^0,w^0)$ is a complex Nash equilibrium and $c_1=c_2=v_A$.
\end{proof}
Theorem \ref{th:7.3} indicates that if we can find an equalizing strategy of player $I$ and an equalizing strategy of player $II$, then these two strategies are optimal and represent a solution of the game. This agrees with the constructive proof of the Minimax Theorem in Section 5. An interesting question now rises: do such equalizing strategies always exist, and if so, how can we find them? The following theorem answers such question.\\ \\
Firstly, the reader may easily observe that the following proposition is true by Definition \ref{def:7.1}:
\begin{proposition}\label{pr:7.4}
Let $G_\mathcal{C}(A)$ be a two-player zero-sum complex game. Player $I$ has an equalizing strategy if and only if the following (linear) system is feasible (has a solution):
\begin{center}
\begin{equation}\label{eq:7.3}
\begin{array}{ll}
    Re(z^*Ad^j)=c_1,\;\;\forall j\in\mathcal{J}\\
    \displaystyle\sum_{i=1}^{m}z_i=1\\
    \lvert argz_i\rvert\leq\alpha,\;\; i=1,2,...,m
    \end{array}
    \end{equation}
\end{center}
where $c_1\in\mathbb{R}$. Correspondingly,  Player $II$ has an equalizing strategy if and only if the following (linear) system is feasible (has a solution):
\begin{center}
\begin{equation}\label{eq:7.4}
\begin{array}{ll}
    Re((d^i)^*Aw)=c_2,\;\;\forall i\in\mathcal{I}\\
    \displaystyle\sum_{j=1}^{n}w_j=1\\
    \lvert argw_j\rvert\leq\beta,\;\; j=1,2,...,n
    \end{array}
    \end{equation}
\end{center}
where $c_2\in\mathbb{R}$.
\end{proposition}\[\]
\begin{theorem}\label{th:7.5}
A two-player zero-sum complex game $G_\mathcal{C}(A)$ has a complex Nash equilibrium in equalizing strategies if and only if both of the following (linear) systems have at least one solution:
\begin{center}
\begin{equation}
\label{eq:7.5}
\begin{array}{ll}
    \displaystyle\sum_{i=1}^{m}\overline{z_i}a_{ij}=\eta,\;\;\;j=1,2,...,n\\
    \displaystyle\sum_{i=1}^{m}\overline{z_i}=1\\
    \lvert arg\overline{z_i}\rvert\leq\alpha,\;\;\;i=1,2,...,m
    \end{array}
    \end{equation}
\end{center}\[\]
\begin{center}
\begin{equation}\label{eq:7.6}
\begin{array}{ll}
    \displaystyle\sum_{j=1}^{n}a_{ij}w_j=\theta,\;\;\;i=1,2,...,m\\
    \displaystyle\sum_{j=1}^{n}w_j=1\\
    \lvert argw_j\rvert\leq\beta,\;\;\;j=1,2,...,n
    \end{array}
    \end{equation}
\end{center}\[\]
where $\eta,\theta\in\mathbb{C}$. If (7.5) and (7.6) have solutions then $Re(\eta)=Re(\theta)=v_A$.
\end{theorem}
\begin{proof}
We show that (\ref{eq:7.5}) is equivalent to (\ref{eq:7.3}) and that (\ref{eq:7.6}) is equivalent to (\ref{eq:7.4}): Relation $Re(z^*Ad^j)=c_1\;\;\forall j\in\mathcal{J}$ is equivalent to the following $n^2$ relations:\\
\begin{equation*}
    d^j=e^1:\;\;\;Re(\overline{z_1}a_{11}+\overline{z_2}a_{21}+...+\overline{z_m}a_{m1})=c_1\;\;\;\;\;\;\;\;\;\;\;\;\;\;\;\;\;\;\;\;\;\;\;\;\;\;\;\;\;\;\;\;\;\;\;\;\;\;\;\;\;\;\;\;\;\;\;\;\;\;\;\;\;\;\;\;\;\;\;\;\;\;\;\;(1)\;\;\;
\end{equation*}
\begin{equation*}
    d^j=e^2:\;\;\;Re(\overline{z_1}a_{12}+\overline{z_2}a_{22}+...+\overline{z_m}a_{m2})=c_1\;\;\;\;\;\;\;\;\;\;\;\;\;\;\;\;\;\;\;\;\;\;\;\;\;\;\;\;\;\;\;\;\;\;\;\;\;\;\;\;\;\;\;\;\;\;\;\;\;\;\;\;\;\;\;\;\;\;\;\;\;\;\;\;(2)\;\;\;
\end{equation*}
\begin{equation*}
    \vdots\;\;\;\;\;\;\;\;\;\;\;\;\;\;\;\;\;\;\;\;\;\;\;\;\;\;\;\;\;\;\;\;\;\;\;\;\;\;\;\;\;\;\;\;\;\;\;\;\;\;\;\;\;\;\;\;\;\;\;\;\;\;\;\;\;\;\;\;\;\;\;\;\;\;\;\;\;\;\;\;\;\;\;\;\;\;\;\;\;\;\;\;\;\;\;\;\;\;\;\;\;\;\;\;\;\;\;\;\;\;\;\;\;\;\;\;\;\;\;\;\;\;\;\;\;\;\;\;\;\;\;\;\;\;\;\;\;
\end{equation*}
\begin{equation*}
    d^j=e^n:\;\;\;Re(\overline{z_1}a_{1n}+\overline{z_2}a_{2n}+...+\overline{z_m}a_{mn})=c_1\;\;\;\;\;\;\;\;\;\;\;\;\;\;\;\;\;\;\;\;\;\;\;\;\;\;\;\;\;\;\;\;\;\;\;\;\;\;\;\;\;\;\;\;\;\;\;\;\;\;\;\;\;\;\;\;\;\;\;\;\;\;\;\;(n)\;\;\;
\end{equation*}
\begin{equation*}
    d^j=\eta^{(n+1)}:\;Re[(\overline{z_1}a_{11}+...+\overline{z_m}a_{m1})(\frac{1}{2}+bi)+(\overline{z_1}a_{12}+...+\overline{z_m}a_{m2})(\frac{1}{2}-bi)]=c_1\;\;\;\;\;\;(n+1)
\end{equation*}
\begin{equation*}
    d^j=\eta^{(n+2)}:\;Re[(\overline{z_1}a_{11}+...+\overline{z_m}a_{m1})(\frac{1}{2}+bi)+(\overline{z_1}a_{13}+...+\overline{z_m}a_{m3})(\frac{1}{2}-bi)]=c_1\;\;\;\;\;\;(n+2)
\end{equation*}
\begin{equation*}
    \vdots\;\;\;\;\;\;\;\;\;\;\;\;\;\;\;\;\;\;\;\;\;\;\;\;\;\;\;\;\;\;\;\;\;\;\;\;\;\;\;\;\;\;\;\;\;\;\;\;\;\;\;\;\;\;\;\;\;\;\;\;\;\;\;\;\;\;\;\;\;\;\;\;\;\;\;\;\;\;\;\;\;\;\;\;\;\;\;\;\;\;\;\;\;\;\;\;\;\;\;\;\;\;\;\;\;\;\;\;\;\;\;\;\;\;\;\;\;\;\;\;\;\;\;\;\;\;\;\;\;\;\;\;\;\;\;\;\;
\end{equation*}
\begin{equation*}
    d^j=\eta^{n^2}:Re[(\overline{z_1}a_{1(n-1)}+...+\overline{z_m}a_{m(n-1)})(\frac{1}{2}-bi)+(\overline{z_1}a_{1n}+...+\overline{z_m}a_{mn})(\frac{1}{2}+bi)]=c_1\;(n^2)
\end{equation*}\\
By (1) and (2) relation (n+1) can also be written as:\\
\begin{equation*}
    \frac{1}{2}c_1-Im(\overline{z_1}a_{11}+\overline{z_2}a_{21}+...+\overline{z_m}a_{m1})b+\frac{1}{2}c_1+Im(\overline{z_1}a_{12}+\overline{z_2}a_{22}+...+\overline{z_m}a_{m2})b=c_1\Leftrightarrow
\end{equation*}
\begin{equation*}
    \Leftrightarrow Im(\overline{z_1}a_{11}+\overline{z_2}a_{21}+...+\overline{z_m}a_{m1})=Im(\overline{z_1}a_{12}+\overline{z_2}a_{22}+...+\overline{z_m}a_{m2})
\end{equation*}\\
Similarly, by all the above relations, we conclude that relations $(n+1),(n+2),...,(n^2)$ are equivalent to:
\begin{equation}\label{eq:7.7}
    Im(\displaystyle\sum_{i=1}^{m}\overline{z_i}a_{ik})=Im(\displaystyle\sum_{i=1}^{m}\overline{z_i}a_{il}):=p_1\;\;\forall k\not =l
\end{equation}
By (\ref{eq:7.7}) and $(1),(2),...,(n)$ we equivalently have the $n$ relations:
\begin{equation*}
    \displaystyle\sum_{i=1}^{m}\overline{z_i}a_{ij}=c_1+p_1i:=\eta\in\mathbb{C},\;\;\; j=1,2,...,n
\end{equation*}\\
Now it is obvious that $\displaystyle\sum_{i=1}^{m}\overline{z_i}=1\Leftrightarrow\displaystyle\sum_{i=1}^{m}z_i=1$ and $\lvert arg\overline{z_i}\rvert\leq\alpha,\;\;\forall i\in\{1,2,...,m\}\Leftrightarrow\lvert argz_i\rvert\leq\alpha,\;\;\forall i\in\{1,2,...,m\}$, hence (\ref{eq:7.5}) is equivalent to (\ref{eq:7.3}).\\ \\
Similarly one can prove that (\ref{eq:7.6}) is equivalent to (\ref{eq:7.4}). Finally, if both (\ref{eq:7.5}) and (\ref{eq:7.6}) have solutions, then $Re(\eta)=c_1=v_A=c_2=Re(\theta)$ by Proposition \ref{pr:7.4} and Theorem \ref{th:7.3}.\\
\end{proof}
\begin{remark}
Note that if (\ref{eq:7.5}) (respectively (\ref{eq:7.6})) has a solution then by (\ref{eq:7.7}) the condition of elimination is always met, meaning that we can search for a solution in any submatrix of $A$ that occurs from iterated eliminations of dominated rows and columns simply by solving (\ref{eq:7.5}), just as in the real case. 
\end{remark}
We end this section with a definition:
\begin{definition}
Let $G_\mathcal{C}(A)$ be a two-player zero-sum complex game. The smallest strategy argument $0<\gamma<\frac{\pi}{2}e$ such that $\underline{h}<\overline{h}$ and the complex systems (7.5), (7.6) are feasible for $\alpha=\beta=\gamma$, is called smallest equalizing argument or lowest equalizing argument (sea/lea).
\end{definition}
Note that such argument does not always exist.\[\]\[\]
\section{Solution of two-player zero-sum complex games}
In this section we describe the process of solving a two-player zero-sum complex game, which is somewhat similar to the one in real space. The following methods are not mentioned or used in \cite{19, 20} during the solution of $2\times2$ matrix games, while we also give a specific example.\[\]

We give the different steps that one should follow in order to solve a two-player zero-sum complex game:\\
\begin{enumerate}
    \item We first check if $\overline{h}=\underline{h}$.  If $\overline{h}>\underline{h}$ then a solution in pure complex strategies does not exist and we move on to the next step.\\
    \item We search for and locate all dominated rows and columns, mainly by Definitions \ref{def:6.3}--\ref{def:6.6}. There are cases where a row of the payoff matrix $A$ is not initially dominated by another row/other rows, but is dominated in a submatrix that occurs from an elimination of a column (same goes with columns). We eliminate such dominated rows and columns until we obtain a submatrix $A'$ where there are no more dominations.\\
    \item We solve the following linear systems in $\mathbb{C}$:
    \begin{center}
\begin{equation}\label{eq:8.1}
\begin{array}{ll}
    \displaystyle\sum_{i=1}^{m}z_ia_{ij}=\eta,\;\;\;j=1,2,...,n\\
    \displaystyle\sum_{i=1}^{m}z_i=1\\
    \end{array}
    \end{equation}
\end{center}
\begin{center}
\begin{equation}\label{eq:8.2}
\begin{array}{ll}
    \displaystyle\sum_{j=1}^{n}a_{ij}w_j=\theta,\;\;\;i=1,2,...,m\\
    \displaystyle\sum_{j=1}^{n}w_j=1
    \end{array}
    \end{equation}\\
\end{center}
for constants $\eta,\theta\in\mathbb{C}$. If (\ref{eq:8.1}) and (\ref{eq:8.2}) have solutions $z^0$ and $w^0$ respectively, then we check if they satisfy:\\
\begin{equation}\label{eq:8.3}
    \lvert argz^0\rvert\leqq\alpha\;\mbox{ and }\;\lvert argw^0\rvert\leqq\beta
\end{equation}\\
If they do satisfy such constraints, then $\overline{z^0},w^0$ are equalizing strategies of players $I$ and $II$ respectively, and since the condition of elimination is met (for every single submatrix during the proccess of iterated elimination) by Theorem \ref{th:7.5} $(\overline{z^0},w^0)$ is a complex Nash equilibrium and $Re(\theta)=Re(\eta)=v_A$. There are many ways of solving complex linear systems of the form $Bz=b$ similar to the real case (e.g. through the augmented complex matrix and row reduction/Gaussian elimination). Note that the existence of such strategies agree with the construcrtive proof of the Minimax Theorem we gave in Section 5.\\
\item If either (\ref{eq:8.1}) or (\ref{eq:8.2}) does not have a solution, or if there exist solutions $z^0,w^0$ that do not satisfy (\ref{eq:8.3}), then we need a different approach: Let $\theta$ and $\eta$ be optimal feasible solutions of problems $(P^*)$ and $(D^*)$ respectively, as they were formulated in the proof of the Minimax Theorem in Section 5
\begin{center}
\begin{equation*}{(P^*)\;\;\;\;\;\;\;\;}
\begin{array}{ll}
    min\;\;Re(e^*z') \\
    s.t.\;\;Az'-e \in(\hat{\mathcal{T}})^*\\
     z'\in\hat{\mathcal{S}}
    \end{array}
    \end{equation*}
\end{center}

\begin{center}
\begin{equation*}{(D^*)\;\;\;\;\;\;\;\;}
\begin{array}{ll}
    max\;\;Re(e^*w') \\
    s.t.\;\;-A^*w'+e\in(\hat{\mathcal{S}})^*\\
     w'\in\hat{\mathcal{T}}
    \end{array}
    \end{equation*}
\end{center}
By the constructive proof of the Minimax Theorem we obtained that $\frac{1}{\underline{v}}=\theta=\eta=\frac{1}{\overline{v}}$. Therefore it suffices to calculate the solutions of the problems $(P^*),\;(D^*)$. However, such a task most of the time requires an exact algorithmic process. Since a precise calculating process, similar to the simplex method, is not yet known in complex space (due to the complexity of the closed convex cones), we suggest a different method.\par
Consider the problems $(\tilde{P}),\;(\tilde{D})$ as they are defined by:
\begin{center}
\begin{equation*}{(\tilde{P})\;\;\;\;\;\;\;\;}
\begin{array}{ll}
    min\;\;Re(e^*z') \\
    s.t.\;\;Az'-e \in(\tilde{\mathcal{T}})^*\\
     z'\in\hat{\mathcal{S}}
    \end{array}
    \end{equation*}
\end{center}

\begin{center}
\begin{equation*}{(\tilde{D})\;\;\;\;\;\;\;\;}
\begin{array}{ll}
    max\;\;Re(e^*w') \\
    s.t.\;\;-A^*w'+e\in(\tilde{\mathcal{S}})^*\\
     w'\in\hat{\mathcal{T}}
    \end{array}
    \end{equation*}
\end{center}
where $\tilde{\mathcal{S}}:=\{z\in\mathbb{C}^{m}:\lvert argz\rvert\leqq\alpha\}\supset\hat{\mathcal{S}}$, $\tilde{\mathcal{T}}:=\{w\in\mathbb{C}^{n}:\lvert argw\rvert\leqq\beta\}\supset\hat{\mathcal{T}}$. Let $\tilde{\theta}$ be an optimal solution of $(\tilde{P})$ and $\tilde{\eta}$ be an optimal solution of $(\tilde{D})$, then $\tilde{\eta}\leq\eta=\theta\leq\tilde{\theta}$. The problems $(\tilde{P}),\;(\tilde{D}),\;\{\tilde{\eta}=\tilde{\theta}\}$ are equivalent to the complex linear complementarity problem (see Definition 2.9):
\begin{equation*}
        M=\begin{pmatrix}
        0& -A& -E& 0\\
        A^*& 0& 0& -E^*\\
        E^*& 0& 0& 0\\
        0& E& 0& 0
        \end{pmatrix},\;\;x=\begin{pmatrix}
        z'\\
        w'\\
        0\\
        0
        \end{pmatrix},\;\;q=\begin{pmatrix}
        e\\
        -e\\
        0\\
        0
        \end{pmatrix},\;\;\gamma=\begin{pmatrix}
        \alpha\\
        \beta\\
        \frac{\pi}{2}e\\
        \frac{\pi}{2}e
        \end{pmatrix}
    \end{equation*}
If this complex LCP has a solution, then $\theta=\tilde{\theta}=\tilde{\eta}=\eta$ and $z',\;w'$ are also the solutions of the problems $(P^*),\;(D^*)$.    
We refer the reader to \cite{14} for a solution method regarding this specific complex LCP.
Note that a solution of the complex LCP does not always exist.\\ \\
\end{enumerate} \par
By following these four steps, a complex Nash equilibrium is not guaranteed since the iterated conditions of elimination might not be met. Therefore, if a solution in equalizing strategies does not exist, we may need to solve the problems $(\tilde{P})$, $(\tilde{D})$ (that is, the equivalent complex LCP through pivoting steps or transforming problems, see \cite{13,14,17,22}) in their initial form, that is, with the initial payoff matrix $A$ instead of the one that occurs from the process of iterated elimination. The reader should also notice that such matrix $A$ must satisfy the initial hypothesis that appears in the proof of the Minimax Theorem, that is, the matrix $A$ may be of the form $A+cE$ for some suitable $c>0$.

\par
For a different approach towards the solution of complex linear programming problems and two-player zero-sum complex games we refer the reader to \cite{21} and \cite{24}.
\[\]
\begin{example}
Let $G_\mathcal{C}(A)$ be a two-player zero-sum complex game, where $A$ is the payoff matrix:\\
\begin{equation*}
    A=\begin{pmatrix}
2 & 1+i & 5+2i\\
3+i & 3 & 4-i
\end{pmatrix}
\end{equation*}\\
and let $S^2_{\frac{\pi}{4}}$, $S^3_{\frac{5\pi}{12}}$ be the strategy sets of players $I$ and $II$ respectively.\\ \\
Firstly, one can see that $\underline{h}<\overline{h}$: the reader may check that $\underline{h}=\frac{(7-\sqrt{3})}{4}<3=\overline{h}$.\\ \\
We then locate all dominated rows and columns of $A$: The third column is weakly dominated by the first column. Indeed we have:
\begin{itemize}
        \item $Re((e^1)^*Ae^3)=5>2=Re((e^1)^*Ae^1)$
        \item $Re((e^2)^*Ae^3)=4>3=Re((e^2)^*Ae^1)$
        \item $Re((\eta^3)^*Ae^3)=6>2=Re((\eta^3)^*Ae^1)$
        \item $Re((\eta^4)^*Ae^3)=3=Re((\eta^4)^*Ae^1)$
\end{itemize}\[\]
Let $A'=\begin{pmatrix}
2 & 1+i\\
3+i & 3
\end{pmatrix}$ \;\;be the matrix that occurs from the elimination of the third column of $A$. We now solve the following complex linear systems:
\begin{center}
\begin{equation}
\begin{array}{ll}
    z_12+z_2(3+i)=\eta\\
    z_1(1+i)+z_23=\eta\\
    z_1+z_2=1
    \end{array}
    \end{equation}
\end{center}
\begin{center}
\begin{equation}
\begin{array}{ll}
    w_12+w_2(1+i)=\theta\\
    w_1(3+i)+w_23=\theta\\
    w_1+w_2=1
    \end{array}
    \end{equation}
\end{center}\[\]
for some $\eta,\theta\in\mathbb{C}$. One can easily find solutions $z^0=\displaystyle(\frac{2}{5}-\frac{i}{5},\frac{3}{5}+\frac{i}{5})^T$ of $(8.4)$ and $w^0=(\displaystyle\frac{4}{5}+\displaystyle\frac{3}{5}i,\displaystyle\frac{1}{5}-\displaystyle\frac{3}{5}i)^T$ of $(8.5)$ that satisfy: $|arg z^0|\leqq\displaystyle\frac{\pi}{4}$ and $|arg w^0|\leqq\displaystyle\frac{5\pi}{12}$. Therefore $(\overline{z^0},w^0)$ is a complex Nash equilibrium of $G_\mathcal{C}(A)$ and the value of the complex game is:
\begin{equation*}
v=Re((\frac{2}{5}-\frac{i}{5})2+(\frac{3}{5}+\frac{i}{5})(3+i))=Re((\frac{4}{5}+\frac{3}{5}i)2+(\frac{1}{5}-\frac{3}{5}i)(1+i))=\displaystyle\frac{12}{5}
\end{equation*}
Note that $v=\displaystyle\frac{12}{5}\in(\underline{h},\overline{h})$.
\end{example}

\section{Games of common argument and Symmetric games}
In this final section we define complex games of common argument and deal with two-player zero-sum symmetric games in the complex case. Such symmetric games are defined differently to \cite{20}.\[\]
\begin{definition}
Let $G_\mathcal{C}=(N,(S^{m_i}_{\alpha_i})_{i\in N},(h_i)_{i\in N})$ be a finite complex game, where $N=\{1,2,...,k\}$. If $\alpha_1=\alpha_2=...=\alpha_k:=\alpha$, then $G_\mathcal{C}$ is a complex game of common argument $\alpha$. \\
\end{definition}
The interpretation of this definition is that in a complex game of common argument all players are only allowed to play strategies that have exactly the same geometric ``bound". That is, all players act fairly between one another and there is no player who is allowed to choose a complex strategy with argument big enough, in relation to the rest, to allow him/her to efficiently improve his/her payoff or force other players to change their own strategies.\[\]
\begin{definition}
A complex game is called fair if the value of the complex game is equal to zero.
\end{definition}
\begin{definition}
A two-player zero-sum complex game $G_\mathcal{C}(A)$ is called square if $A\in\mathbb{C}^{n\times n}$ for some $n\geq 2$.
\end{definition}
\begin{definition}
A two-player zero-sum complex game $G_\mathcal{C}(A)$ is called skew-hermitian if it is square and $A=-A^*$.
\end{definition}
\begin{definition}\label{def:9.5}
A skew-hermitian complex game $G_\mathcal{C}(A)$ of common argument $\alpha$ is called symmetric.
\end{definition}
The property of common argument is necessary in order for Definition \ref{def:9.5} to be the generalization of the one given in real space. That is, in a two-player zero-sum symmetric complex game players $I$ and $II$ can swap identities any time, without anyone gaining advantage or disadvantage against the other. The players need to have the same strategy argument in order for this symmetric property to hold, as shown in the following.\[\]
\begin{proposition}
Let $G_\mathcal{C}(A)$ be a two-player zero-sum symmetric complex game with common strategy argument $\alpha$. Then a row $k$ of $A$ is weakly (strictly) dominated by row $l$ if and only if column $k$ is weakly (strictly) dominated by column $l$.
\end{proposition}
\begin{proof}
We only show the result in the case of weak domination (strict domination is similarly proved). Note that $\mathcal{I}=\mathcal{J}$ since players $I$ and $II$ have the same strategy argument and the game is square. Row $k$ is weakly dominated by row $l$ if and only if:\\
\begin{equation*}
    Re((e^k)^*Ad^j)\leq Re((e^l)^*Ad^j)\;\;\forall j\in\mathcal{J}\;\mbox{ and }\exists j_0\in\mathcal{J}: \;Re((e^k)^*Ad^{j_0})<Re((e^l)^*Ad^{j_0})\;\Leftrightarrow
\end{equation*}
\begin{equation*}
    \Leftrightarrow Re(\overline{(e^k)^*Ad^j})\leq Re(\overline{(e^l)^*Ad^j})\;\;\forall j\in\mathcal{J}\;\mbox{ and }\exists j_0\in\mathcal{J}: \;Re(\overline{(e^k)^*Ad^{j_0}})<Re(\overline{(e^l)^*Ad^{j_0}})\;\Leftrightarrow
\end{equation*}
\begin{equation*}
\Leftrightarrow Re((d^j)^*A^*e^k)\leq Re((d^j)^*A^*e^l)\;\;\forall j\in\mathcal{J}\;\mbox{ and }\exists j_0\in\mathcal{J}: \;Re((d^{j_0})^*A^*e^k)<Re((d^{j_0})^*A^*e^l)\;\Leftrightarrow
\end{equation*}
\begin{equation*}
\Leftrightarrow -Re((d^j)^*A^*e^k)\geq -Re((d^j)^*A^*e^l)\;\forall j\in\mathcal{J}\;\mbox{and }\exists j_0\in\mathcal{J}: -Re((d^{j_0})^*A^*e^k)>-Re((d^{j_0})^*A^*e^l)
\end{equation*}
\begin{equation*}
\Leftrightarrow Re((d^j)^*Ae^k)\geq Re((d^j)^*Ae^l)\;\;\forall j\in\mathcal{J}\;\mbox{ and }\exists j_0\in\mathcal{J}: \;Re((d^{j_0})^*Ae^k)>Re((d^{j_0})^*Ae^l)
\end{equation*}\\
if and only if column $k$ is weakly dominated by column $l$.
\end{proof}\[\]
\begin{theorem}
Let $G_\mathcal{C}(A)$ be a two-player zero-sum symmetric complex game of common argument $\alpha$. Then:\\
\begin{enumerate}
    \item If $z^0$ is an optimal strategy for player $I$ then it is also an optimal strategy for player $II$. Correspondingly, if $w^0$ is an optimal strategy for player $II$ then it is also an optimal strategy for player $I$.
    \item If $(z^0,w^0)$ is a complex Nash equilibrium then $(w^0,z^0)$, $(z^0,z^0)$, $(w^0,w^0)$ are also complex Nash equilibria.
    \item $G_\mathcal{C}(A)$ is a fair complex game, that is $v=0$.
\end{enumerate}
\end{theorem}
\begin{proof}
\begin{enumerate}
    \item Let $z^0$ be an optimal strategy for player $I$ and $w^0$ an optimal strategy for player $II$ (such strategies always exist by the Minimax Theorem). Then, $(z^0,w^0)$ is a complex Nash equilibrium and we therefore have:\\
    \begin{equation*}
        Re(z^*Aw^0)\leq Re((z^0)^*Aw^0)\leq Re((z^0)^*Aw)\;\;\forall z,w\in S^{m}_\alpha\;\Leftrightarrow
    \end{equation*}
    \begin{equation*}
        \Leftrightarrow Re((w^0)^*A^*z)\leq Re((w^0)^*A^*z^0)\leq Re(w^*A^*z^0)\;\;\forall z,w\in S^{m}_\alpha\;\Leftrightarrow
    \end{equation*}
    \begin{equation*}
        \Leftrightarrow -Re((w^0)^*A^*z)\geq -Re((w^0)^*A^*z^0)\geq -Re(w^*A^*z^0)\;\;\forall z,w\in S^{m}_\alpha\;\Leftrightarrow
    \end{equation*}
    \begin{equation*}
        \Leftrightarrow Re((w^0)^*Az)\geq Re((w^0)^*Az^0)\geq Re(w^*Az^0)\;\;\forall z,w\in S^{m}_\alpha
    \end{equation*}\\
    if and only if $(w^0,z^0)$ is a complex Nash equilibrium. Therefore if $z^0$ is an optimal strategy for player $I$ then it is also an optimal strategy for player $II$, and if $w^0$ is an optimal strategy for player $II$ then it is also an optimal strategy for player $I$.\\
    \item If $(z^0,w^0)$ is a complex Nash equilibrium then $(w^0,z^0)$ is also a complex Nash equilibrium by $(1)$, and therefore by Corollary \ref{cor:4.29} $(z^0,z^0)$, $(w^0,w^0)$ are also complex Nash equilibria.\\
    \item Let $(z^0,w^0)$ be a complex Nash equilibrium. Then by $(2)$ $(z^0,z^0)$ is also a complex Nash equilibrium, hence:
    \begin{equation*}
        v=Re((z^0)^*Az^0)=Re(\overline{(z^0)^*Az^0})=Re((z^0)^*A^*z^0)=Re((z^0)^*(-A)z^0)=-v\Rightarrow v=0
    \end{equation*}
\end{enumerate}
\end{proof}
\[\]

\section*{Acknowledgements}
We would like to express our sincere appreciation to professor Solan E. for his extremely important and enlightening comments as well as to professor Melolidakis C. for introducing us to Game Theory and Linear Programming. We would also like to show our deep gratitude to mathematician and mentor Margaritis A. for his invaluable guidance and encouragement throughout these three and a half years of our under--graduate studies.
\[\]\[\]

\bibliographystyle{plain}

\end{document}